\documentclass[12pt]{amsart}

\newcommand{\la}{\langle}
\newcommand{\ra}{\rangle}

\newcommand{\CC}{{\mathbb C}}
\newcommand{\RR}{{\mathbb R}}
\newcommand{\NN}{{\mathbb N}}

\newcommand{\sgn}{\operatorname{sgn}}
\newcommand{\Res}{\operatorname{Res}}
\renewcommand{\Re}{\operatorname{Re}}
\renewcommand{\Im}{\operatorname{Im}}

\newcommand{\sech}{\operatorname{sech}}

\newcommand{\defeq}{\stackrel{\rm{def}}{=}}
\def\squarebox#1{\hbox to #1{\hfill\vbox to #1{\vfill}}}
\newcommand{\stopthm}{\hfill\hfill\vbox{\hrule\hbox{\vrule\squarebox
                 {.667em}\vrule}\hrule}\smallskip}
\newcommand{\ev}{\textnormal{ev}}

\usepackage{lscape}
\usepackage{amssymb}
\usepackage{amsthm}
\usepackage{amsxtra}
\usepackage{graphicx}
\usepackage{color}

\newtheorem{theorem}{Theorem}
\newtheorem{definition}[theorem]{Definition}
\newtheorem{proposition}{Proposition}[section]
\newtheorem{lemma}[proposition]{Lemma}
\newtheorem{corollary}[proposition]{Corollary}

\theoremstyle{remark}
\newtheorem{remark}[proposition]{Remark}

\numberwithin{equation}{section}

\ifx\pdfoutput\undefined
  \DeclareGraphicsExtensions{.pstex, .eps}
\else
  \ifx\pdfoutput\relax
    \DeclareGraphicsExtensions{.pstex, .eps}
  \else
    \ifnum\pdfoutput>0
      \DeclareGraphicsExtensions{.pdf}
    \else
      \DeclareGraphicsExtensions{.pstex, .eps}
    \fi
  \fi
\fi

\title[Breathing patterns in nonlinear relaxation]
{Breathing patterns in nonlinear relaxation}

\author[J. Holmer]
{Justin Holmer}
\email{holmer@math.brown.edu}
\address{Department of Mathematics, Brown University\\
151 Thayer Street, Providence, RI 02912, USA}
\author[M. Zworski]
{Maciej Zworski}
\email{zworski@math.berkeley.edu}
\address{Mathematics Department, University of California \\
Evans Hall, Berkeley, CA 94720, USA}

\begin{document}

\begin{abstract} In numerical experiments involving nonlinear solitary 
waves propagating
through nonhomogeneous media one observes ``breathing'' in the sense
of the amplitude of the wave going up and down on a much faster scale 
than the motion of the wave 
-- see Fig. \ref{f:fastslow1} below. 
In this paper we investigate
this phenomenon in the simplest case of stationary waves in which 
the evolution corresponds to relaxation to a nonlinear ground state.
The particular model is the popular $ \delta_0 $ impurity in the 
cubic nonlinear Schr\"odinger equation on the line.
We give asymptotics of the amplitude on a finite but relevant time interval
and show their remarkable agreement with numerical experiments, see Fig. \ref{f:1}.
 We stress
the nonlinear origin of the ``breathing patterns'' caused by
the selection of the ground state depending on the initial data, and by 
the non-normality of the linearized operator.
\end{abstract}

\maketitle


\section{Introduction}
\label{int}

We study a simple model of relaxation to a nonlinear ground state.
Our equation is the one dimensional nonlinear cubic 
Schr\"odinger equation with a small delta potential:
\begin{equation}
\label{eq:nls}
i\partial_t u + \tfrac{1}{2}\partial_x^2 u + q\delta_0(x)u +u|u|^2 = 0 \,,
\end{equation}
where $ 0  < |q| \ll 1 $. The nonlinear ground state minimizes the 
corresponding energy \eqref{eq:GPH}
for a prescribed $ L^2 $ norm, and is explicitly given by 
\begin{equation}
\label{E:gs}
v_\lambda ( x ) = \lambda \sech ( \lambda | x | 
+ \tanh^{-1}( q / \lambda ) ) \,, \ \
\| v_\lambda \|_2^2 = 2 ( \lambda - q ) \,, \ \ \lambda > |q| \,. 
\end{equation}
A simple rescaling allows the reduction to the case $ \lambda = 1 $ and we obtain
\begin{theorem}
\label{T:main}
Suppose that $ u ( x, t  ) $ solves \eqref{eq:nls}, 
$ u( x , 0 ) \in H^1 ( \RR ) $ is real and even, 
and that 
\begin{equation}
\label{eq:t00}
 \| x^k \partial^\ell_x  w_0 \|_{L^\infty ( ( 0 , \infty ) ) } 
\leq C_{kl} |q| \,, \ \ 
w_0 ( x ) \defeq u ( x , 0 ) - v_1 ( x ) \,, \ \ 
k \,, \ell \in \NN \,.   \end{equation}
Then for 
$ \lambda = 1 +  \int_\RR  w_0 ( x ) v_1( x) dx  $ and 
$ 0 \leq t \ll  |q|^{-1/2} $, 
we have 
\begin{equation}
\label{eq:t0} 
\|  u (  x, t  ) - e^{ it\lambda^2 /2 } \Big( v_\lambda ( x ) 
+ w ( \lambda x, \lambda^2 t ) \Big) \|_{H_x^1} 
\leq C |q|^{3/2}  + C t^2 q^2 \,, 
\end{equation}
where $ w ( x , t ) $ is given  explicitly in \eqref{eq:lw0q}. In particular,
for $ 1 \ll t \leq C |q|^{-2/7} $, 
\begin{equation}
\label{eq:t3}
u (  0, t  ) = e^{ it \lambda^2/2 } \left( \lambda  - 
\sqrt {\frac{ 2 }{ \pi t} } 
e^{ i (  \lambda^2 t /2 +  \pi/4 ) } \int_\RR w_0( x ) dx  \right) 
+ {\mathcal O} \left( \frac  q {t^{3/2} } \right)   
 \,. \end{equation}
\end{theorem}

\begin{figure}
\begin{center}
\includegraphics[width=6in]{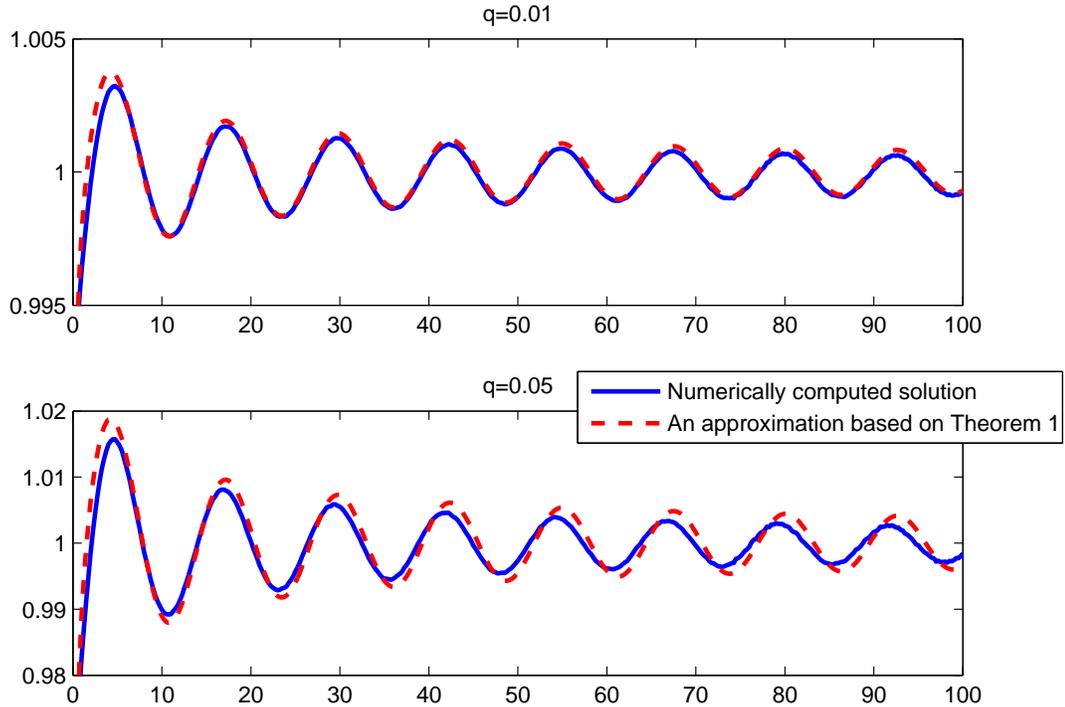}
\end{center}
\caption{Breathing patterns for $ u( x , 0 ) = \sech(x/(1+q))/(1+q) $ (the initial
data is rescaled so that the ground state to which it relaxes is $ v_1 ( x ) $): 
the plots show $ |u ( 0 , t )| $ and the asymptotic 
prediction \eqref{eq:t3} given in Theorem \ref{T:main}, 
for $ q = 0.05 $ and $ q = 0.01 $. The agreement is
remarkably good for times much longer than given in the theoretical 
result.
}
\label{f:1}
\end{figure}

The conditions on $ w_0 $ in \eqref{eq:t00} can be weakened considerably,
and in particular we only need estimates for $ k \,, \ell \leq N $ for
some $ N$. Theorem \ref{T:npt} below gives a statement which depends only on 
$ \| w_0 \|_{H^1} $ being small and the more explicit results in Theorem 
\ref{T:main} come from our close analysis of the propagator $ \exp ( - i t 
{\mathcal L}_{q,\lambda} ) $ appearing in \eqref{eq:t1}. As explained
below our motivation comes from the study of solitons and the 
intial data in which we are most interested is $ u  ( x , 0 ) = \sech x $.
For that case, the comparison of the theorem with numerical results is shown in 
Fig.\ref{f:1}.

The reduction to the case $\lambda=1$, mentioned before the statement
of Theorem \ref{T:main} is straightforward:
let $\tilde u(x,t) = \lambda^{-1} u(\lambda^{-1} x, \lambda^{-2} t)$.  
Then $\tilde u$ solves \eqref{eq:nls} with $q$ replaced by $\lambda^{-1} q$, 
$i\partial_t \tilde u + \tfrac{1}{2}\partial_x^2 \tilde u 
+ q\lambda^{-1} \delta_0(x)\tilde u +\tilde u|\tilde u|^2 = 0$.
Now, if we suppose the theorem holds when applied to $\tilde u$ replacing 
$u$ and $q/\lambda$ replacing $q$, then we can deduce the theorem in its 
current form.  Thus, it suffices to prove the $\lambda =1$ case.

In the remainder of the introduction we will discuss our motivation and 
relations to existing literature, possible approaches to obtaining 
finer asymptotics, and a simple example of a breathing pattern for nonnormal
operators.

\subsection{Motivation} 
Mathematical studies of relaxation to ground states for nonlinear Schr\"odinger 
equations have been recently conducted in a number of mathematical 
papers, see Soffer-Weinstein \cite{SoWe}, Tsai-Yau \cite{TsYa}, 
Gang-Sigal \cite{GaSi}, Gang-Weinstein \cite{GaWe}, and references given 
there. The particular focus is on the behaviour as $ t \rightarrow \infty $
(genuine relaxation in the sense of pure mathematics)
and the allowed non-linearities typically exclude standard examples from the 
physical literature. For the cubic nonlinear Schr\"odinger equation (NLS) 
on the line, that is for \eqref{eq:nls} with $ q = 0 $, the nonlinear
relaxation can be studied in great detail using methods of inverse 
scattering theory pioneered by Zakharov-Shabat \cite{ZS72} -- see
Deift-Its-Zhou \cite{DIZ},  Deift-Zhou \cite{DZ}, 
and \cite[Appendix B]{HMZ1} for recent advances
and references. The case of $ q \neq 0 $ with even initial data 
is also in principle 
accessible by these methods as was pointed out by Fokas \cite{Fo}.

\begin{figure}
\begin{center}
\includegraphics[width=6in]{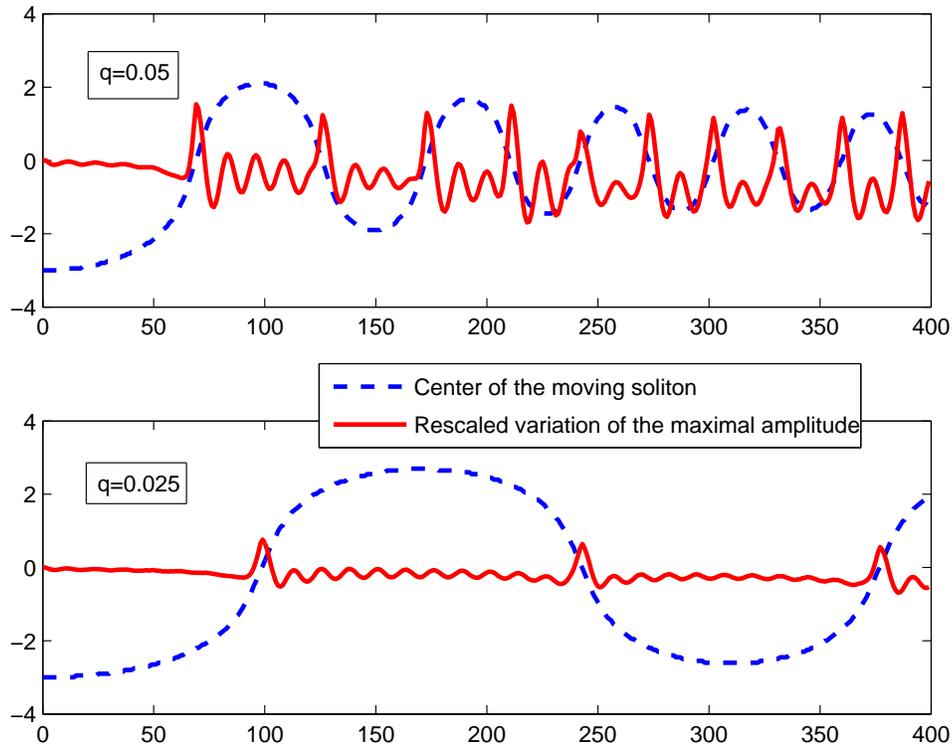}
\end{center}
\caption{Fast and slow oscillations in the motion of the soliton
with initial condition $ \sech ( x + 3 ) $ moving in the field
of $ - \delta_0 ( x )/20 $ (top figure) and $ - \delta_0 ( x ) / 40 $ (bottom 
figure). The center of the soliton oscillates
around $ x = 0 $ with a much larger period than the rescaled
amplitude $ A ( t ) = 30 ( |u ( a(t) , t ) | -1 ) $, where $ a ( t ) $ 
is the center of the moving of the moving soliton. The periods of the fast
oscillations are close.}
\label{f:fastslow1}
\end{figure}

The goals of this paper are more modest: we explain 
a phenomenological fact occuring on shorter time scales
for a simple physically relevant model of NLS with small $ \delta $ impurities. 
Numerous references for this model in the physics 
literature can for instance be found in 
\cite{BL} (where it is used to model more realistic narrow traps), 
\cite{CM},\cite{GHW}, and \cite{LFDKS}. See also \cite{Sac} for a recent
numerical study and further pointers to the literature.

Our motivation came from observing a common phenomenon illustrated in 
Fig.\ref{f:fastslow1}. In \cite{HZ1} we have shown that the solution of
\eqref{eq:nls} with $ u ( x , 0 ) = \sech( x -a_0) e^{i x v_0} $ 
satisfies 
\begin{equation}
\label{eqo:t1}
\| u ( t ,\bullet ) -  e^{i(\bullet -a(t)) v(t) }e^{i\gamma ( t ) } \sech
(\bullet -a ( t) ) \|_{H^1 ( \RR ) } \leq C |q|^{1-3 \delta} \,,
\end{equation}
for $ 0 < t < \delta (v_0^2+|q|)^{-1/2}\log(1/|q|) $, $ 0 < |q| \ll 1 $, 
and 
where $ a$, $ v$, and $ \gamma$ solve the following system
of equations
\begin{gather}
\label{eqo:t3}
\begin{gathered}
\frac{d}{dt} {  a} =   v \,, \ \
\frac{d}{dt} {  v} = \frac12 q\partial_x (\sech^2) (  a) \,,
\\
\frac{d}{dt} {  \gamma} = \frac12 + \frac {v^2} 2
+ q \sech^2 (   a ) + \frac12 q a \partial_x (\sech^2) (  a) \,,
\end{gathered}
\end{gather}
with initial data $(a_0,v_0,0)$ (please note that the sign convention for 
$ q $ has been changed here). As was pointed out there, as seen from 
explicit constants in coercive estimates, these asymptotics require
$ q \lesssim 0.01 $ to hold accurately. From the semiclassical point
of view $ q = h^2 $, where $ h$ is the {\em effective Planck constant}
of the problem, so that means $ h \lesssim 0.1 $ -- see \cite{HZ2} for 
an explanation of this scaling philosophy. 

In Fig.~\ref{f:fastslow1} the dashed line shows the motion of the
center of the soliton in the case of $ q = 0.05 $ (which is a borderline
case for the applicability of \eqref{eqo:t1}). We see oscillations
with the period proportional to $ q^{-1/2} $ in agreement with 
\eqref{eqo:t3}. The continuous line shows the oscillation of 
the amplitude: we look at the deviations of the value of the solution
at the maximum of $ |u ( x , t )| $ in $ x $ from $ 1 $, the maximal 
value of the absolute value of the soliton solution. The oscillations
are much faster than the oscillations of the center of the soliton
and the period is close to being fixed. Numerical observations
suggest that the period is almost independent of $ q $.

This ``breathing'' behaviour is even more striking in movies
of numerical solutions (see for instance the last movie in 
{\tt http://math.berkeley.edu/$\sim$zworski/msg.pdf}). The slowing
down of the soliton and the shedding of its mass seem
closely related to these breathing patterns.

As the first step to understand solitons moving in 
nonhomogeneous media we study the stationary case, that is \eqref{eq:nls}
with initial data given by $ u ( x , 0 ) = \sech x $. The results of 
\cite{HZ1} recalled in \eqref{eqo:t1} and \eqref{eqo:t3} show that
\[   \| u ( t ,\bullet ) -  e^{i\gamma ( t ) } \sech
(\bullet)  \|_{H^1 ( \RR ) } \leq C |q|^{1-3 \delta} \,, \]
for $ 0 < t < \delta |q|^{-1/2} \log ( 1/|q|) $, 
and $ \gamma ( t ) = ( 1/2 +q ) t $. In fact, an application of the 
method of \cite{HZ1} shows that for some $ \tilde \gamma ( t ) $, 
\[   \| u ( t ,\bullet ) -  e^{i\tilde \gamma ( t ) } \sech
(\bullet)  \|_{H^1 ( \RR ) } \leq C |q| \,, \]
for all times. Here we could replace $ \sech $ with $ v_\lambda $ for
any $ \lambda = 1 + {\mathcal O} ( q ) $.

Hence the breathing patterns must involve higher 
order asymptotics and since $ |q|= h^2 $, the  natural next
step is $ |q|^{3/2} = h^3 $. Theorem \ref{T:main} provides 
that next step on a time scale which allows seeing a large number
of oscillations. The numerical experiments show a very good agreement with 
asymptotics provided by \eqref{eq:t3} and suggest that they are valid
for times longer than $ t \ll |q|^{-1/2}$. 

Finer asymptotics might be possible if one adapts some of the methods
of \cite{SoWe}, \cite{GaSi}, and \cite{GaWe},
but it is not clear which direction should be taken for the efficient study of
moving solitons. We opted for the simplest at this early stage.

\subsection{Nonlinear aspects of  ``breathing''}
\label{nab}

We first compare the breathing patterns observed here with 
amplitude oscillations in the relaxation to the ground state of 
a linear problem, $ i u_t = -u_{xx}/2 - q \delta_0 u $, $ 0 < q \ll 1 $.
If the initial data is equal to $ u_0 ( x ) $ and is real and even,
a heuristic approximation for the solution is 
\[ u ( 0 , t ) \sim  e^{ it q^2/2 }  q  \int_\RR u_0 ( x ) e^{-q|x|} dx + 
\frac{ \hat u_0 ( 0 )} {\sqrt t } \,. \]
Although the zero resonance disappears (see \cite{S} and \S \ref{sa}), 
for $ q $ small we expect 
the behaviour $ 1/\sqrt t $ to  persist for long times -- see \S \ref{aabp}.

\begin{figure}
\begin{center}
\includegraphics[width=4in]{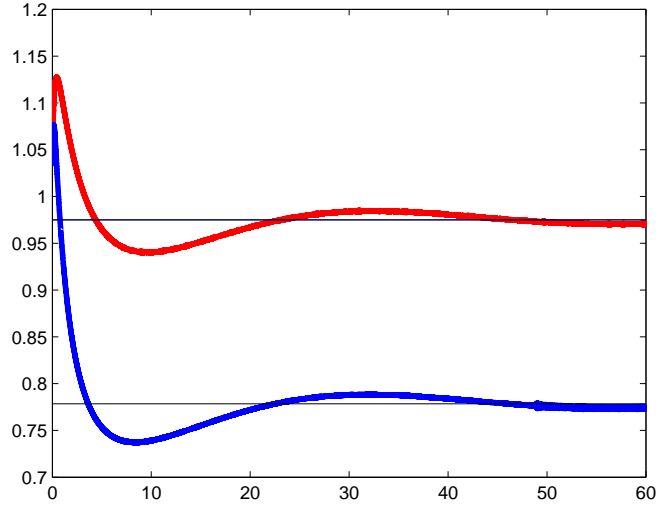}
\end{center}
\caption{Examples of linear relaxation: the initial condition is
$ u ( x , 0 ) = \sech x $ and the potentials are $ - q \delta_0 ( x)  $ 
with$ q = 1/2 $
for the {\color{red} top} graph and $ q = \sqrt 2 / 4 $ for the
{\color{blue} bottom} graph. We expect the periods of oscillations
to be $ 4 \pi/ q^2 $. Consequently,
changing $ q = 1/2 $ to $ q /\sqrt 2 = \sqrt 2 / 2 $ we expect the period
to double and to see that we plot $ {\color{red} | u ( 0 , t ) | }$
in the top graph and $ {\color{blue} | u ( 0 , t/2 ) | } $ in the bottom
graph: the agreement of the periods is striking.
The horizontal lines correspond to the asymptotic values
$ v_q ( 0 )  \int \sech x v_q ( x ) dx $, $ v_q = \sqrt{q} \exp ( - q|x|) \
$.
}
\label{f:pcool}
\end{figure}

This is very different from \eqref{eq:t3}. The main difference is
that in the nonlinear problem the eigenvalue is not fixed but
it is selected depending on the initial condition. The approximate
selection is given by the formula for $ \lambda $ in Theorem \ref{T:main}
and a more precise selection method is given in 
Proposition \ref{p:eigen-select} preceding Theorem \ref{T:npt}. 
In particular, the periods oscillation for a fixed initial condition are 
approximately independent of $ q $. Since the linear eigenvalue is 
fixed and depends on $ q $ this is strikingly different in linear relaxation
-- see Fig.\ref{f:pcool}.

The origin of the phase in the second term in \eqref{eq:t3} lies
in the properties of the non-normal linearized operator for \eqref{eq:nls},
and in particular in the coupling responsible for the nonnormality.
We do not yet have a fully conceptual explanation for that other
than the analysis of \eqref{eq:lw0q}. 

We present a simple example illustrating how non-normality can 
be responsible for ``breathing'', that is oscillations in the 
amplitute, absent for normal operators.
Suppose that $\alpha,\beta \in \mathbb{R}$ (note that we allow both 
positive and negative $\alpha$, $\beta$), $\vec w= [\Re w \; \Im w]^T$, 
\begin{equation}
\label{E:Rtoy}
R = \begin{bmatrix} 0 & -\beta-\partial_x^2 \\ \alpha+\partial_x^2 & 0 
\end{bmatrix}
\end{equation}
and we consider the evolution 
\begin{equation}
\label{E:Rtoy-evolve}
\partial_t \vec w = R \vec w
\end{equation}  
We consider $R$ as an operator on 
$H^2(\mathbb{R};\mathbb{C}) \times H^2(\mathbb{R};\mathbb{C})$ 
and write out explicit formulas for the complex-valued vector plane-wave 
solutions associated to (generalized) eigenvalues $\pm i\omega$ of $R$.  From 
these, we build real-valued plane-wave solutions $\vec w=[w_1 \; w_2]^T$ to 
the matrix equation \eqref{E:Rtoy-evolve}.  When these are converted to 
complex numbers as $w=w_1+iw_2$, we find that if $\alpha=\beta$, then $w$ is 
unimodular but if $\alpha\neq \beta$, then $w$ is not unimodular and we see 
oscillations in amplitude.

Let $\omega \geq 0$.  We seek complex-valued (vector) plane-wave solutions to 
\eqref{E:Rtoy-evolve} with generalized eigenvalue $\pm i\omega$.  Let 
$\gamma \geq \max(\sqrt{\max(\alpha,0)}, \sqrt{\max(\beta,0)})\geq 0$ be the 
unique solution to 
$$\omega^2=(\gamma^2-\alpha)(\gamma^2-\beta) \,.$$
Set 
$$\sigma = \sqrt{ \frac{\gamma^2-\alpha}{\gamma^2-\beta} } \,.$$
Now let
$$v(\gamma) = \begin{bmatrix} 1 \\ i\sigma \end{bmatrix} 
e^{i\gamma x}, \qquad \tilde v(\gamma) 
= \begin{bmatrix} 1 \\ -i\sigma \end{bmatrix} e^{i\gamma x} \,.$$
Then $v(\gamma)$, $v(-\gamma)$ are two plane-wave solutions to $Rv = i\omega v$
and $\tilde v(\gamma)$, $\tilde v(-\gamma)$ are two plane-wave solutions to 
$R\tilde v = -i\omega \tilde v$.  From this we see that
$$e^{-it\omega}v(\pm \gamma), \qquad e^{it\omega}\tilde v(\pm \gamma)$$
are four solutions to \eqref{E:Rtoy-evolve}.   Thus
$$ \vec w= \begin{bmatrix} \cos(t\omega + \gamma x) \\ 
\sigma \sin(t\omega + \gamma x) \end{bmatrix} 
= \tfrac12 e^{-it\omega}v(-\gamma) + \tfrac12 e^{it\omega}\tilde v(\gamma)$$
is a real solution.  Forming a complex number from this vector, we obtain
$$ \cos(t\omega + \gamma x) + i\sigma  \sin(t\omega + \gamma x)$$
which has constant-in-time modulus if and only if $\sigma =1$.  

The linearization of \eqref{eq:nls} around the solution $v_1$ gives the non-normal operator $\mathcal{F}_q$ defined below in \eqref{eq:Fql}.  The motivation for studying the operator $R$ above is that $\mathcal{F}_q$ has the form of $R$ with $\alpha=\beta=-1$ for $|x|$ large but with $\alpha=5$, $\beta=1$ for $|x|$ near $0$.

\subsection{Method of proof and organization of the paper}

The proof Theorem \ref{T:main} consists of two parts. 
The first is a 
nonlinear perturbation theory presented in Theorem \ref{T:npt} and
gives an approximation of the solution by the linearized flow.
The second part is 
the precise analysis of that linearized flow on time scales consistent
with the approximation given in Theorem \ref{T:npt}.

In \S \ref{pr} we present various standard facts about the 
nonlinear Schr\"odinger flow with an external delta function 
potential. The Hamiltonian structure of this flow with respect to 
the symplectic form $ \omega  ( u , v ) = \Im \int u \bar v $ is
particularly crucial. It plays an important r\^ole in \S \ref{npt}
where it is used to select the nonlinear eigenvalue of the limiting
``relaxed state''. The other component of the proof of Theorem \ref{T:npt},
which is the main result of that section, 
is the coercivity estimate allowing the control of $ H^1 $ norm 
by the linearized operator. The estimates on the propagator are
similar to 
the estimates in \S 5 of \cite{HZ2}:
instead of the $ L^2 $-energy method we estimate $ \partial_t \langle 
{\mathcal L}_q u , u \rangle $, where $ {\mathcal L}_q $ is essentially
the Hessian of the Hamiltonian (see \eqref{E:Lq}). The initial data
is assumed to be even which is particularly important in the case of
$ q < 0 $ (repulsive $\delta$ potential) as the ground state is then 
unstable -- see \cite{LFDKS}.

In \S \ref{fc} we recall Kaup's explicit spectral decomposition of 
the linearized operator for the focusing cubic NLS on the line. 
As shown in Appendix \ref{ndkp} that basis
can also be discovered via simple numerical experimentation. 
We use it to obtain a representation of the propagator and apply it
to see the relaxation to solitons in the free case. For initial 
data close to the solitons this crude approximation is remarkably 
close to the precise results given by full inverse spectral method
-- see Fig.\ref{f:2}.

The results of \S \ref{fc} lead to an almost explicit spectral decomposition
for the operator with $ 0 < |q| \ll 1 $ -- that is presented in \S \ref{sedp}
and Appendices \ref{S:pnz} and \ref{AA}. We follow the general theory 
of Buslaev-Perelman and Krieger-Schlag in that particular setting. More
general non-linearities could also be allowed but since we are ultimately
interested in the comparison with numerics the explicit nature of Kaup's
basis is very useful. As in \S \ref{fc} this leads to a representation of
the propagator, \eqref{eq:lw0q}, used in the statement of Theorem \ref{T:main}.
An asymptotic analysis of \S \ref{aabp}  gives the approximation \eqref{eq:t3} illustrated in Fig.\ref{f:1}.


\medskip

\noindent
{\sc Acknowledgments.}
We would like to thank Michael Weinstein and Galina Perelman for 
stimulating conversations and e-mail exchanges.
The work of the first author was supported in part by an NSF postdoctoral
fellowship, and that
of the second second author by the NSF grant DMS-0654436. 

\section{Preliminaries}
\label{pr}

In this section we review various basic aspects of the equation 
\eqref{eq:nls}.

\subsection{Hamiltonian structure} 
The nonlinear Schr\"odinger equation 
\eqref{eq:nls} describes the Hamiltonian flow on 
on $ H^1 ( \RR, \CC ) $ for the Hamiltonian 
\begin{equation}
\label{eq:GPH}
H_q ( v ) \defeq \frac14 \int ( |\partial_x v |^2 - |v|^4 ) dx
- \frac12 q | v ( 0 ) |^2 \,. \end{equation}
More precisely, we consider  
\[ V = H^1 ( \RR, \CC ) \simeq  H^1 ( \RR , \RR) \oplus H^1 ( \RR, \RR ) \,, \ \
u \simeq ( \Re u , \Im u ) \,, \] 
as a {\em real} Hilbert space with 
the inner product and the symplectic form given by 
\begin{equation}
\label{eq:omega}\
\langle u,v \rangle  \defeq \Re \int u\bar v \,, \  \ 
\omega(u,v) \defeq \langle i u , v \rangle = \Im \int u\bar v\,, 
\end{equation}
Let $ H_q $ given by \eqref{eq:GPH}, or be a more general function,
$H: V \to \mathbb{R}$. 
The associated Hamiltonian
vector field is a map $\Xi_H : V\to TV$, 
which means that for a particular point $u\in V$, we have 
$(\Xi_H)_u \in T_uV$. The vector field  $\Xi_H$ is defined by the relation
\begin{equation}
\label{eq:Hamvf} \omega(v , (\Xi_H)_u) = d_uH(v)
\,, \end{equation}
where $v\in T_uV$, and $d_uH:T_uV \to \mathbb{R}$ is defined by
$$d_uH(v) = \frac{d}{ds}\Big|_{s=0} H(u+sv) \,. $$
In the notation above
\begin{equation}
\label{eq:nats}  dH_u ( v ) = \langle dH_u , v \rangle \,, 
\ \  (\Xi_H)_u = \frac 1 i dH_u  \,. 
\end{equation}

For $ H = H_q $ given by \eqref{eq:GPH} we compute
\begin{align*}
d_uH(v) &= \Re \int ( (1/2) \partial_x u \partial_x \bar v - |u|^2u \bar v ) dx 
- \Re (q  u ( 0 ) \bar v  )\\
&= \Re \int ( - (1/2)\partial_x^2 u - |u|^2u - q \delta_0 ( x )  u )\bar v \,. 
\end{align*}
Thus, in view of \eqref{eq:nats} and \eqref{eq:Hamvf}, 
$$(\Xi_H)_u = \frac 1 i \left( - \frac12 \partial_x^2u - |u|^2u 
- q \delta_0 ( x ) u \right) $$
The flow associated to this vector field (Hamiltonian flow) is
\begin{equation}
\label{eq:Hflow}
\dot u = (\Xi_H)_u =  \frac 1 i  \left( - 
\frac12 \partial_x^2u - |u|^2u - q \delta_0 ( x ) u \right) 
\,.
\end{equation}

\subsection{Well posedness in $ H^1$}

The discussion here has been formal but it is well known that 
the equation \eqref{eq:nls} has global solutions in $ H^1 $ for 
more general nonlinearities, $ |u|^{p-1} u $, $ 1 < p < 5 $. 
For the reader's convenience we recall the standard argument.

We have the following basic estimates:
\begin{equation}
\label{eq:basic}  \| u \|_{L^\infty }^2 \leq C
\| u \|_{L^2} \| u' \|_{L^2}  \,, 
\end{equation}
(which follows from the fundamental theorem of calculus: $ u ( x ) ^2 
= \int_{-\infty}^x 2 u ( y ) u'( y ) dy $) and thus
\[ \frac1{p+1}\int | u|^{p+1}  \leq   \frac1{p+1}\| u \|_{L^\infty}^{{p-1} } \| u \|^2_{L^2}
\leq C \| u' \|_{L^2}^{\frac{p-1}2} \| u\|^{\frac{p+3}2 }_{L^2}
 \leq
 \frac1{16} \| u' \|_{L^2}^2  +  C' \| u \|_{L^2}^{\frac{2(p+3)}{5-p}} \,,\]
$$\frac{q}{2}|u(0)|^2 \leq \frac1{16}\|u'\|_{L^2}^2+Cq^2\|u\|_{L^2}^2 \,.$$
Hence, 
\begin{align*}
H_q(u) &= \frac14\|u'\|_{L^2}^2 - \frac1{p+1}\|u\|_{L^{p+1}}^{p+1}-\frac{q}{2}|u(0)|^2\\
&\geq \frac18\|u'\|_{L^2}^2 - C\|u\|_{L^2}^{\frac{2(p+3)}{5-p}} - Cq^2\|u\|_{L^2}^2
\end{align*}
and consequently, 
\[ \| u \|_{H^1}^2 \leq 8 H_q( u ) + C  \| u \|_{L^2}^{\frac{2(p+3)}{5-p}} 
+ (C q^2+1) \| u \|_{L^2}^2 \,.  \]
Since the energy, $ H_q(u) $, and mass, $ \| u \|_{L^2 }^2 $, are conserved,
we see that if the solution exists in $ H^1 $, its $H^1 $
norm is uniformly bounded. Thus we only need to show {\em local} existence
in $ H^1 $.
Let us fix $ T > 0 $ and, for $ u = u ( x , t ) $, define the norm
\[ \| u \|_X \defeq \sup_{ 0 \leq t \leq T } \| u ( \bullet , t ) \|_{H^1}
\,. \]
Solving \eqref{eq:nls}, $ u ( x , 0 ) = u_0 ( x ) $, 
is equivalent to finding the fixed point of the operator
\[  \Phi \; : \; u ( x, t ) \longmapsto e^{ i t (  \partial_x^2/2 + q \delta_0 ( x ) ) } u_0 ( x )
- \frac 1 i \int_0^t e^{ i ( t- s) (\partial^2_x / 2 + q \delta_0 ( x ) ) } 
( | u |^{p-1} u )  ( x , s ) ds \,. \]
Here the operator $ \exp ( i t  (\partial^2_x / 2 + q \delta_0 ( x ) ) ) $
is unitary on $ L^2 $ (see the discussion of the operator $ L $ given in
\eqref{eq:L} below) and preserves 
$$\tilde H_q(u) = \frac14\|u'\|_{L^2}^2 - \frac{q}{2}|u(0)|^2 \,.$$ 

Again using \eqref{eq:basic}, 
$$ \frac18\|u'\|_{L^2}^2 - Cq^2\|u\|_{L^2}^2 \leq \tilde H_q(u(t)) \leq \frac12\|u'\|_{L^2}^2 + Cq^2\|u\|_{L^2}^2 \,.$$
Therefore, if $u(t) =  \exp ( i t  (\partial^2_x / 2 + q \delta_0 ( x ) ) )u_0$, 
\begin{align*}
\frac18 \|u'(t)\|_{L^2}^2 &\leq H_q(u(t)) + Cq^2\|u(t)\|_{L^2}^2 \\
&= H_q(u_0) + Cq^2\|u_0\|_{L^2}^2 \\
&\leq \frac12\|u_0'\|_{L^2}^2 + Cq^2\|u_0\|_{L^2}^2
\end{align*}
From this, we see that 
\[ 
e^{ i t  (\partial^2_x / 2 + q \delta_0 ( x ) )  }
\; : \;  H^1 ( \RR ) \; \longrightarrow H^1 \; ( \RR ) \,, \]
is bounded with norm independent of $t$.
This and the estimate
\[ \begin{split}
 \| | u |^{p-1} u - | v|^{p-1}  \|_{H^1 } & \leq C ( \| |u|^{p-1} \|_{H^1} 
+ \| |v|^{p-1}  \|_{H^1 }  ) \| u - v \|_{H^1 } \\
& \leq C
( \| u \|_{H^1} + \| v \|_{H^1} )^{p-1}  \| u - v \|_{H^1} \,. 
\end{split} \]
give
\[  \| \Phi ( u ) - \Phi ( v ) \|_X  \leq  C T ( \| u \|_X + \| v \|_X)^{p\
-1}
( \| u - v \|_X ) \]
so that for $ T $ small fixed point arguments can be used to obtain
a solution in $ H^1 $. 

\subsection{Nonlinear ground states}
The minimizers with a prescribed $ L^2 $ norm are given by 
critical points of the Hamiltonian with the constraint added (and 
$ \lambda^2/4 $ playing the r\^ole of the Lagrange multiplier):
\begin{equation}
\label{eq:Eql}
  {\mathcal E}_{q,\lambda} ( u ) \defeq H_q ( u ) + \frac{\lambda^2}4 
\| u \|^2_{L^2} \,.\end{equation}
Then for the ground state given by \eqref{E:gs} we obtain
\[  {\mathcal E}_{q, \lambda}' ( v_\lambda ) = 0 \,, \ \ 
{\mathcal E}_{q, \lambda }'' ( v_\lambda ) = {\mathcal L}_{q}  \,, \]
where the Hessian is the following self-adjoint operator on 
$ H^1 ( \RR , \CC ) \simeq H^1 ( \RR, \RR ) \oplus H^1 ( \RR , \RR ) $:
\begin{equation}
\label{E:Lq}  {\mathcal L_q } \defeq \begin{bmatrix} L_{q+} & 0 \\
0 & L_{q-} \end{bmatrix} \,, 
\end{equation}
where
\[\begin{aligned}
& L_{q+} = \frac12(\lambda^2 -\partial_x^2 - 6v^2 -2q\delta_0 ) \,,  \\
& L_{q-} = \frac12 (\lambda^2 -\partial_x^2 - 2v^2 -2q\delta_0 )  \,.
\end{aligned}
\] 
In view of the $ \delta $ functions,
 the definition of the operators $ L_{q\pm} $ is given by choosing
the correct domain for the operator.
To see what it is
let us first examine the basic case of 
\begin{equation}
\label{eq:L}
{L}=-\partial_x^2 + V - q\delta_0\,, 
\end{equation} 
on $\mathbb{R}$, where $V$ is a 
smooth real-valued potential, rapidly decaying at $\infty$.  Suppose that 
$u\in L^2(\RR)$, ${L} u =f$, and $f\in L^2(\RR)$.  This implies that 
away from $0$ we have that $\partial_x^2 u \in L^2$, and thus 
$u\in H^2(\mathbb{R}\backslash\{0\})$.  In order that $f$ remain a function 
across $x=0$, we must have that $u(x)$ is continuous at $x=0$ and 
\begin{equation}
\label{E:jump1}
u'(0+)-u'(0-) = -qu(0)
\end{equation}
Thus a natural domain to consider for ${L}$ is
\begin{equation}
\label{E:DL}
\mathcal{D} = \{ \, u \, | \, u\in H^2(\RR\backslash \{0\})\, , 
u \text{ is continuous at }x=0 \text{ and \eqref{E:jump1} holds} \}
\end{equation}
By verifying that the operators $ L \pm i $ are both symmetric and surjective 
on $ {\mathcal D} $ we see that $ L $ is self-adjoint with domain 
$\mathcal{D}$.

\subsection{Linearization and the Hamiltonian map}
\label{S:lHm}

For $ {\mathcal E} : V \rightarrow \RR $ satisfying $ {\mathcal E}'(u) = 0 $
we can invariantly define the {\em Hamiltonian map}, 
\[ {\mathcal F}\; : \; T_u V \longrightarrow T_u V \,, \]
using the well defined Hessian of $ {\mathcal E} $ at $ u$:
\[ \langle {\mathcal E}'' ( u ) X , Y \rangle = \omega ( Y , {\mathcal F} X) \,.\]
In other words, 
the Hamiltonian map is the {\em linearization} of the Hamilton vector field
of $ {\mathcal E}$. See for instance 
\cite[Sect.21.5]{Hor2} for a general discussion, and 
\cite[Lemmas 2.1, 2.2]{HZ2} for relevant facts in our context.

For $ V = H^1 ( \RR , \CC ) $ with the symplectic form
 \eqref{eq:omega} we have 
\[ {\mathcal F} = -i {\mathcal E}'' \,, \]
and for $ {\mathcal E} $ given by \eqref{eq:Eql} we have 
\begin{equation}
\label{eq:Fql}  {\mathcal F}_q = -i {\mathcal L}_q = 
 \begin{bmatrix} 0 & L_{q-} \\
-  L_{q+} & 0 
 \end{bmatrix} \,. 
\end{equation}
The matrix representation is based on the identification 
\[ H^1 ( \RR , \CC ) \ni u \simeq [ \Re u , \Im u ]^t \in H^1 ( \RR, \RR)^2 \,,\]
It is also convenient to consider the equivalent matrix representation 
using the identification,
\[ H^1 ( \RR , \CC ) \ni u \simeq [ u , \bar u ]^t \in \Delta \subset 
H^1 ( \RR, \CC)^2 \,, \]
which gives
\begin{gather}
\label{eq:Hql}
\begin{gathered}
\frac12 H_q \defeq \frac 1 {i}U {\mathcal F}_q U^* \,, \ \ U \defeq \frac 1 {\sqrt{2}}
\begin{bmatrix} 1 & -i \\ 1 & \ i \end{bmatrix} \,, \\
H_0 =  \begin{bmatrix} - \partial^2_x + \lambda^2 & \ \ \ 0 \\
\ \ \ 0 & \partial_x^2 - \lambda^2 \end{bmatrix} + \sech^2 x
\begin{bmatrix} -4 & -2 \\
\ \ 2 &  \ \ 4\end{bmatrix} \,, \\
H_{q,\lambda} = H_q =  \begin{bmatrix} - \partial^2_x + \lambda^2 & \ \ \ 0 \\
\ \ \ 0 & \partial_x^2 - \lambda^2 \end{bmatrix} + v_{\lambda}^2 (x)
\begin{bmatrix} -4 & -2 \\
\ \ 2 &  \ \ 4\end{bmatrix} 
- 2 q \delta_0 
\begin{bmatrix} 1 & \ 0 \\
0 &  -1 \end{bmatrix} 
\end{gathered}
\end{gather}
(when there is no, or little, chance of confusion we supress $ \lambda $ in 
our notation; most of the time its value is taken to be $ 1$).
This representation is convenient when we study the 
spectral decomposition of $ F_q $. The factor $ \frac 12 $ was introduced to 
make the notation simpler and to have a better agreement with the standard
notation of \cite{Kaup76}, \cite{BP}, \cite{KS}.

Since the energy 
$ H_q (u )  $ differs from $ 2 {\mathcal E}_{q,\lambda} $ by the additive
mass term, $ \lambda^2 \| u \|^2/2 $,
 these linearizations differ from the linearization of
\eqref{eq:nls} by a constant only. 

For future reference we also note the symmetries of $ H_q $. 
Let
$ \sigma_j $ be the Pauli matrices,
\begin{equation}
\label{eq:Pauli} \sigma_1 \defeq \begin{bmatrix} 0 & 1 \\ 1 & 0 \end{bmatrix} \,, \ \
 \sigma_2 \defeq \begin{bmatrix} 0 & -i \\ i & 0 \end{bmatrix} \,, \ \
 \sigma_3 \defeq \begin{bmatrix} 1 & 0 \\ 0 & -1  \end{bmatrix} \,.
\end{equation}
We recall that they 
are characterized by the properties that $\sigma_j^2 = I$ and 
$\sigma_j^*=\sigma_j$.  Using this notation, 
\begin{equation}
\label{eq:Hsym}
\sigma_1 H_q \sigma_1 = -H_q,  \qquad \sigma_3 H_q \sigma_3 = H_q^*\,.
\end{equation}

General considerations show that
$\sigma(H_q) \subset \mathbb{R} \cup i \RR $, and in fact 
$\sigma(H_q)  \setminus \sigma_{\rm pp} ( H_q) = 
(-\infty, -1] \cup [1,+\infty)$. 
The fact that all pure point spectrum is contained in 
$ \RR \cup i \RR $  follows by 
examining the squared operator $H_q ^2$ which turns out to be self-adjoint.
(see Buslaev-Perelman \cite[\S 2.2.3]{BP}). 
This can be done despite the fact that the operator contains the $ \delta_0 $
potential. To see that consider again the operator $ L $ given in 
\eqref{eq:L} and 
suppose that we want to consider ${L}^2$, the squared operator.  
Away from $x=0$, we see that we must have 
$u\in H^4(\mathbb{R}\backslash \{0\})$.  If ${L}u=f$, then we need 
$f\in \mathcal{D}$.  Since $f$ is continuous at $0$, we see from the equation 
${L}u=f$ that 
$\lim_{x\to 0-} \partial_x^2 u(x) = \lim_{x\to 0+} \partial_x^2 u(x)$.  
Moreover, taking 
$$u''(0) \defeq \lim_{x\to 0-} \partial_x^2 u(x) 
= \lim_{x\to 0+} \partial_x^2 u(x)$$
implies 
$$u''(0) = V(0)u(0)-f(0)$$
Away from $x=0$, we have $-u'''+V'u+Vu'=f'$ and thus the condition that
$f'(0+)-f'(0-)=-qf(0)$ becomes
\begin{equation}
\label{E:jump2}
u'''(0+)-u'''(0-)=-qu''(0)
\end{equation}
Define
\begin{equation}
\label{E:DL2}
\tilde{\mathcal{D}} = \{ \,  u \, | \, 
u\in H^4(\RR\backslash \{0\})\, , u, u'' \text{ are continuous at }x=0 
\text{ and \eqref{E:jump1},\eqref{E:jump2} hold} \}
\end{equation}
Provided $u\in \tilde{\mathcal{D}}$, ${L}^2u$ is defined and 
belongs to $L^2(\mathbb{R})$ -- so there is no need 
to worry about the square of the 
delta function not being defined.  Indeed, as soon as we know that 
$u\in \tilde{\mathcal{D}}$ as defined here, then one need only compute 
$(-\partial_x^2+V)^2u$ away from $x=0$ to obtain ${L}^2u$.  
It is thus natural to consider the squared operator $H_q^2$ on $\tilde{\mathcal{D}}\times \tilde{\mathcal{D}}$, and $H_q^2$ can in fact be shown to be self-adjoint on this domain.

\subsection{Symmetries and the generalized kernel}
\label{sgk}

As in \cite{HZ1} and \cite{HZ2} it is convenient
to introduce a natural group action on $ H^1 $:
\begin{gather}
\label{eq:repG}
\begin{gathered}    H^1 \ni u \longmapsto g\cdot u \in H^1 \,, \ \ 
(g\cdot u)(x) \defeq e^{i\gamma}e^{iv(x-a)}\mu u(\mu(x-a)) \,, \\
g = ( a, v , \gamma, \mu ) \in \RR^3 \times \RR_+ \,.
\end{gathered}
\end{gather}
This action gives a group structure on $ \RR^3 \times \RR_+ $
and it is easy to check that this
transformation group is a semidirect product of the Heisenberg group
$ H_3 $ and $ \RR_+ $:
\[ G= H_3\ltimes\mathbb{R}_+ \,, \ \ 
\mu\cdot(a,v,\gamma) = (\frac{a}{\mu} ,\mu v, \gamma) \,.\]


\begin{figure}
\begin{center}
\includegraphics[width=6in]{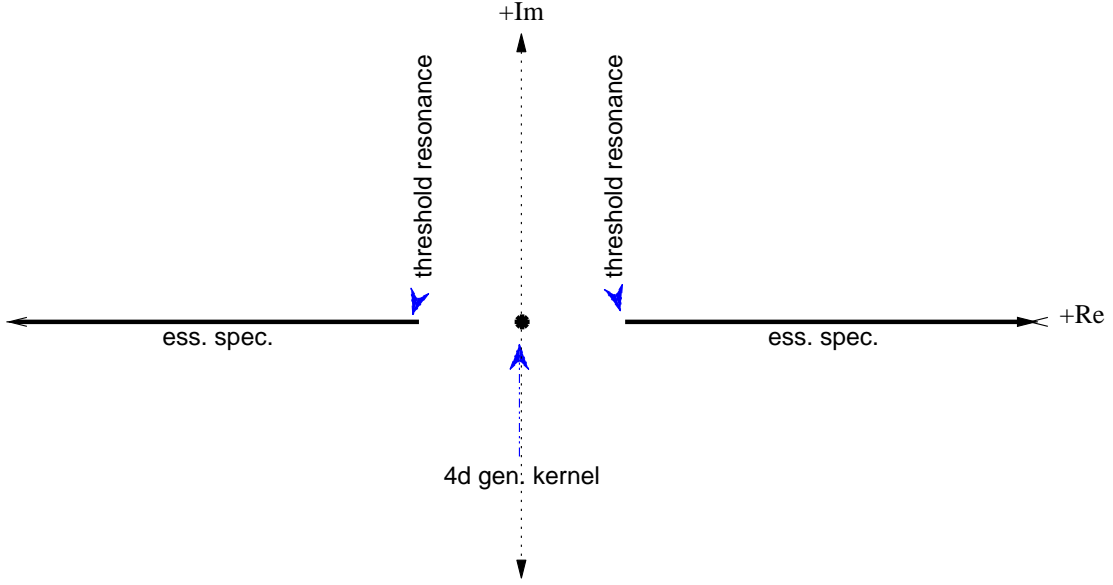}
\end{center}
\caption{The spectrum of the operator $ H_0 $. In the notation 
of \S \ref{sgk}, the generalized eigenspace at $ 0 $
is spanned by $ e_j \cdot \sech $, $ j=1,\cdots,4$.}
\label{f:pcool2}
\end{figure}

The Lie algebra of $ G$, denoted by $ {\mathfrak g } $, 
is generated by $ e_1, e_2, e_3 , e_4 $, 
which in the infinitesimal representation obtained from \eqref{eq:repG} 
is given by
\begin{equation}
\label{eq:liea1}  e_1 = -\partial_x \,, \ \
e_2 = ix \,,  \ \ e_3 = i \,, \ \  e_4 = \partial_x \cdot x \,. 
\end{equation}
It acts, for instance, on $ {\mathcal S}( \RR ) \subset H^1 $, 
and by $ X \in {\mathfrak g} $ we will denote a linear combination of
the operators $ e_j $. We note
that for $ q = 0 $ (and hence $ v = v_\lambda = \lambda \sech ( \lambda x ) $)
\begin{gather}
\label{eq:omegae}
\begin{gathered}
\omega ( e_1 \cdot v , e_2 \cdot v ) = 1 \,, \ \
\omega ( e_3 \cdot v , e_4 \cdot v ) = 1 \,, \\
\omega ( e_j \cdot v , e_3 \cdot v ) = \omega ( e_j \cdot v , e_4 \cdot v ) = 
0 \,, \ \ j = 1, 2 \,. 
\end{gathered}
\end{gather}
In the case of $ q = 0 $ the Hamilton vector fields, $ \Xi_{H_0} $, 
$ \Xi_{{\mathcal E}_{0, \lambda}} = \Xi_{H_0} + \lambda/4 $, are tangent 
to the {\em manifold of solitons}, $ G \cdot v $ -- see 
\cite[\S 3]{FrSi} or \cite[\S 2.2]{HZ2}. Hence $ {\mathcal F}_0 = 
- i {\mathcal L}_0 $
preserves $ T_v ( G \cdot v ) \simeq {\mathfrak g} \cdot v $. In fact, 
$ {\mathfrak g} \cdot v $ is the generalized kernel of $ - i {\mathcal L}_0 $:
\begin{equation}
\label{E:kernel1}
\begin{split}
& i\mathcal{L}_0(e_1\cdot v) = 0\,, \qquad i\mathcal{L}_0(e_2\cdot v) = e_1\cdot v\,,
\\ 
& i\mathcal{L}_0(e_3\cdot v) = 0\,, \qquad i\mathcal{L}_0(e_4\cdot v) = e_3\cdot v\,.
\end{split}
\end{equation}
The first and third equation are an immediate consequence of the invariance
of solutions under the circle action $ u \mapsto e^{i\theta } u $, 
and translation in $ x $. 

When $ q \neq 0 $ we lose the translation invariance but we still have 
$ i \mathcal{L}_q ( e_1 \cdot v ) = 0 $ due to the preserved circle action
symmetry. As shown in Appendix \ref{S:pnz} the generalized kernel is given by 
$$V_0 \defeq \operatorname{span}\{ v_3,v_4\}\,, $$
where 
$$v_3(x) = i v_\lambda(x) \Big|_{\lambda=1} \,, \quad v_4(x) = \partial_\lambda \Big|_{\lambda=1} v_\lambda(x), \ \ 
i\mathcal{L}_qv_3=0 \,, \ \ i\mathcal{L}_q v_4=v_3 \,. $$
Hence $ v_j $, $ j = 3, 4 $, are the generalizations of $ e_j \cdot v $, 
and in fact, 
\[ v_3 = e_3 \cdot v \,, \ \ \ 
 v_4 = e_4 \cdot v + {\mathcal O}_{H^1} ( q) \,, \]

\subsection{Coercivity estimate}
Finally we recall the crucial coercivity estimate which in a more 
general form is well known since the work of Weinstein \cite{We}.
For the special case at hand an elementary presentation can be found
in \cite[\S 4]{HZ1}.

For $ q = 0 $ we have the following estimate:
Let $ w \in H^1 ( \RR, \CC ) $ and
suppose that for any $ X \in {\mathfrak g}$,
$ \omega ( w , X \cdot \eta ) = 0 $.
Then,
\begin{equation}
\label{eq:coer}
\begin{split}
 \langle {\mathcal L}_0 w , w \rangle  \geq
 \frac{2 \rho_0  } { 7 + 2 \rho_0 }  \| w \|^2_{H^1}
\simeq 0.0555 \| w \|_{H^1}^2 \,, \ \ \rho_0 = \frac{ 9 }{ 2 ( 12 + \pi^2 \
)} \,.
\end{split}
\end{equation}

\section{Nonlinear perturbation theory}
\label{npt}

In this section we prove a result describing eigenstate selection 
and nonlinear flow approximation for a time depending on the initial
data and on the size of $ q $. Although we restrict our attention to 
the physical (and completely integrable) case of the cubic
NLS the arguments apply to nonlinearities for which the Weinstein coercivity
conditions are satisfied (see \cite{We} and Lemma \ref{L:spec_lower_bd} 
below). 

Recall
$$v_{\lambda,q}(x) = \lambda \sech(\lambda |x|+\tanh^{-1}(q/\lambda)) \,,
\ \ \| v_{\lambda, q} ( x ) \|_{L^2}^2 = 2 ( \lambda - q ) \,.   
$$
Define the projection 
\begin{equation}
\label{eq:proj}
P_{\lambda,q}\varphi \defeq \omega(\varphi,\partial_\lambda v_{\lambda,q})iv_{\lambda,q} - \omega(\varphi,iv_{\lambda,q})\partial_\lambda v_{\lambda,q}
\end{equation}
onto the generalized kernel 
$$V_{\lambda,q} \defeq \operatorname{span}_{\mathbb{R}}\{ iv_{\lambda, q}, \partial_\lambda v_{\lambda,q}\}$$
of $\mathcal{L}_{\lambda, q}$.
We will only use the $q$ subscript when it is needed for clarity (e.g. in the scaling argument below).  We will also drop the $\lambda$-subscript when $\lambda=1$.  Recall that $\la u, v \ra = \Re \int u \bar v$.

\begin{proposition}[Symplectic orthogonality]
\label{p:eigen-select}
There exists $\delta>0$ such that the following holds.  If $\varphi\in H^1$ and there exists $\lambda_0>0$, $\theta_0\in \mathbb{R}$ such that $\|\varphi-e^{i\theta_0}v_{\lambda_0}\|_{H^1} \leq \delta$, then there exists $\lambda\in (0,+\infty)$, $\theta\in \mathbb{R}$  such that $P_{\lambda}(e^{-i\theta}\varphi-v_\lambda) =0$. 
\end{proposition}
\begin{proof}
Let $F:H^1\times (0,+\infty)\times \mathbb{R}\to \mathbb{R}^2$ be given by
$$F(u_0,\lambda, \theta) = \begin{bmatrix} \omega(u_0 - e^{i\theta}v_\lambda, \; ie^{i\theta}v_\lambda) \\ \omega(u_0 - e^{i\theta}v_\lambda, \; e^{i\theta}\partial_\lambda   v_\lambda) \end{bmatrix}$$
Fix $\theta_0$, $\lambda_0$.  Note that $F( e^{i\theta_0}v_{\lambda_0}, \lambda_0, \theta_0)=0$ and the matrix 
$$[ \partial_\lambda F \quad \partial_\theta F ] = 
\begin{bmatrix} 1 & 0 \\ 0 & -1 \end{bmatrix}$$
is (uniformly in $\lambda$, $\theta$) nondegenerate at $u_0=e^{i\theta_0}v_{\lambda_0}$, $\lambda=\lambda_0$, $\theta=\theta_0$.  The implicit function theorem completes the proof.
\end{proof}

\begin{theorem}[Nonlinear perturbation theory]
\label{T:npt}
Let $I\Subset (0,+\infty)$ and $|q|\ll 1$.  Suppose that $u(x,t)$ is an even solution to \eqref{eq:nls} and $w_0(x)\defeq u(x,0)-e^{i\theta}v_\lambda(x)$ satisfies $\|w_0\|_{H^1} \leq h\ll 1$ and $P_{\lambda} (e^{-i\theta}w_0)=0$ for some $\lambda\in I$, $\theta\in \mathbb{R}$.  Then
\begin{equation}
\label{eq:t1} 
\|  u (t) - e^{ it\lambda^2 /2 } \Big( v_\lambda  
+ e^{-it\mathcal{L}_{\lambda,q}}w_0 \Big) \|_{H_x^1} 
\leq C t(1+t)h^2  \,,  
\end{equation}
for all $0\leq t\ll  h^{-1/2}$.
\end{theorem}
\begin{remark}
The constant $C$ depends on $I$ (the range of values in which $\lambda$ lies), and the restrictions $|q|\ll 1$, $h\ll 1$ and $t\ll h^{-1/2}$ all indicate an implicit (small) constant depending on $I$.  
\end{remark}

\begin{remark}  We will ultimately take $h=Cq$ to prove Theorem 1 in \S \ref{ptm}.  Note that our use of $h$ here is different from the connection to the semiclassical problem (discussed in the introduction) where $q=h^2$.
\end{remark}

\begin{lemma}[Coercivity]
\label{L:spec_lower_bd}
There exists $c_0>0$ (independent of $q$) with the following property:  If $|q|\ll 1$, 
$P_q f=0$ and $f$ is even, then 
$$\|f\|_{H_x^1}^2 \leq c_0\la \mathcal{L}_q  f, f \ra$$
\end{lemma}
\begin{proof}
We have, with $f=f_1+f_2$
$$\la \mathcal{L}_q f,f \ra = \la L_{q+}f_1,f_1\ra + \la L_{q-}f_2,f_2 \ra \,,$$
where $L_{q+}$ and $L_{q-}$ are the self-adjoint operators defined in \eqref{E:Lq}.  It suffices to prove that if $f$ is even and real-valued, then
\begin{equation}
\label{E:co1}
\la f, v \ra = 0 \implies \la L_+f,f \ra \geq c\|f\|_{L^2}^2,
\end{equation}
and
\begin{equation}
\label{E:co2}
\la f, \partial_\lambda v_\lambda |_{\lambda=1} \ra =0 \implies \la L_-f,f \ra \geq c\|f\|_{L^2}^2 \,.
\end{equation}

The operators $L_\pm$ (defined as $L_{q\pm}$ with $q=0$) were analysed in \cite[\S 4]{HZ1}, and it was proved there that $\sigma(L_+) = \{ -\frac32,0\} \cup [\frac12,+\infty)$ and $\sigma(L_-)=\{ 0\} \cup [\frac12,+\infty)$.  Moreover, the eigenvalues and $L^2$ normalized eigenfunctions are given explicitly:
$$L_+ (\tfrac{\sqrt 3}{2}\sech^2) = -\tfrac32 (\tfrac{\sqrt 3}{2}\sech^2), \qquad L_+ (\sqrt{\tfrac32}\sech') = 0, \qquad L_-(\tfrac1{\sqrt 2}\sech) =0 \,.$$
By perturbation theory (this is standard perturbation theory for 2nd-order scalar self-adjoint operators, as opposed to the perturbation theory of Appendix \ref{S:pnz}), $\sigma(L_{q+}) = \{ \lambda_1,\lambda_0\} \cup [\frac12,+\infty)$, where $\lambda_1=-\frac32+\mathcal{O}(q)$ and $\lambda_0 = \mathcal{O}(q)$.  Moreover, the $L^2$ normalized associated eigenfunctions, $g_1$ and $g_0$,
$$L_{q+}g_1=\lambda_1 g_1, \qquad L_{q_+}g_0=\lambda_0 g_0\,,$$
satisfy
$$g_1(x)=\tfrac{\sqrt 3}{2}\sech^2x + \mathcal{O}(q), \qquad g_0(x)=\sqrt{\tfrac32}\sech'x + \mathcal{O}(q) \,.$$
In particular, $g_0(x)$ is ``nearly'' odd.  However, we also have that $L_{q+}g_0(-x)=\lambda_0g_0(-x)$, and hence $g_0(-x) = cg(x)$ for some constant $c$.  If $g_0(0) \neq 0$, then $c=1$ and hence $g_0(x)$ is even which contradicts the fact that it is nearly odd.  From this we conclude $g_0(0)=0$, and since $g_0(x)$ is not identically zero and solves a second order ODE, we must have $g_0'(0)\neq 0$.  Taking the derivative of the identity $g_0(-x) = cg_0(x)$ and evaluating at $x=0$ gives that $c=-1$.  Hence $g_0$ is exactly odd.

Now we prove \eqref{E:co1}.  Since $f$ is assumed even, we have by the Spectral Theorem that if $\la f, g_1 \ra =0$, then $\la L_{q+} f,f\ra \geq \frac12 \|f\|_{L^2}^2$.  By \cite[Lemma 4.2]{HZ1} (with, in the notation of that Lemma, $v_0=g_1$, $v_1=v$, $c_0=-\lambda_1$), we have that if $\la f, v\ra =0$, then 
$$\la L_{q+}f , f \ra \geq \left( \lambda_1+(\tfrac12-\lambda_1)\frac{\la g_1,v\ra^2}{\|v\|_{L^2}^2}\right) \|f\|_{L^2}^2 \,.$$
By the perturbation theory, the coefficient evaluates to $\frac{3\pi^2}{16}-\frac32+\mathcal{O}(q) >0$.

Now we carry out the analysis of $L_{q-}$.  By perturbation theory, we know that $\sigma(L_{q-}) = \{ \lambda_2 \} \cup [\frac12,+\infty)$, where $\lambda_2=\mathcal{O}(q)$.  However, direct calculation shows that $L_{q-}v=0$ (where $v(x) = \sech(|x|+\tanh^{-1}q)$), and thus $\lambda_2=0$.  By the Spectral Theorem, if $\la f, v\ra =0$, then $\la L_{q-}f,f\ra \geq \frac12 \|f\|_{L^2}^2$.   By \cite[Lemma 4.2]{HZ1},  we have that if $\la f, \partial_\lambda v_\lambda|_{\lambda=1}\ra =0$,
 $$\la L_{q-}f , f \ra \geq \left( \tfrac12\frac{\la \partial_\lambda v_\lambda |_{\lambda=1},v\ra^2}{\|\partial_\lambda v_{\lambda}|_{\lambda=1}\|_{L^2}^2\|v\|_{L^2}^2}\right) \|f\|_{L^2}^2 \,.$$
Elliptic regularity completes the argument.
\end{proof}

\begin{proof}[Proof of Theorem \ref{T:npt}]
Let $\tilde u(x,t) = e^{-i\theta}\lambda^{-1}u(\lambda^{-1}x,\lambda^{-2}t)$.  Then $\tilde u$ solves
$$i\partial_t\tilde u + \tfrac12\partial_x^2\tilde u + \frac{q}{\lambda}\delta_0\tilde u + |\tilde u|^2\tilde u =0$$
(\eqref{eq:nls} with $q$ replaced by $q/\lambda$) with initial data $\tilde u_0(x) = e^{-i\theta}\lambda^{-1} u_0(\lambda^{-1} x)$.  Moreover,
$$\|u_0-e^{i\theta}v_{\lambda,q}\|_{H^1} \leq h \implies \|\tilde u_0 - v_{1,q/\lambda}\|_{H^1}\leq C ( \lambda )  h$$ 
and 
$$P_{\lambda,\theta,q}(u_0-e^{i\theta}v_\lambda)=0\implies P_{1,0,q/\lambda} (\tilde u_0-v_{1,q/\lambda}) = 0$$ 
Hence, it suffices to prove the theorem in the case $\lambda=1$, $\theta=0$.  Let
$$v_3(x) = i v_\lambda(x) \Big|_{\lambda=1} \,, \quad v_4(x) = \partial_\lambda \Big|_{\lambda=1} v_\lambda(x)$$
Note (by direct computation) that $i\mathcal{L}_qv_3=0$ and $i\mathcal{L}_q v_4=v_3$ and
$$Pw = \omega(w,v_4)v_3-\omega(w,v_3)v_4$$ 
is the symplectic orthogonal projection onto the generalized kernel $V_0=\operatorname{span}\{v_3,v_4\}$.
  
Define $v(x)\defeq v_\lambda(x)\big|_{\lambda=1}$ and $w(t)$ by the relation $u(t)=e^{it/2}(v+w(t))$, and then note that $w$ solves
$$
\left\{
\begin{aligned}
&\partial_t w = -i\mathcal{L}_q w + iF\\
&w\big|_{t=0} = u_0-v
\end{aligned}
\right. \, ,
$$
where
$$F= 2v|w|^2+vw^2+|w|^2w \,.$$
Also define 
$$w_1\defeq e^{-\frac12 it\mathcal{L}_q}(u_0-v) 
\quad \text{and} \quad \tilde w \defeq w-w_1 \,,$$  
so that $\tilde w$ satisfies 
$$
\left\{
\begin{aligned}
&\partial_t \tilde w = - i \mathcal{L}_q \tilde w + iF \\
&w\big|_{t=0}=0
\end{aligned}
\right. \, ,
$$
where now we write $F$ as

$$F = 2v|w_1+\tilde w|^2+v(w_1+\tilde w)^2
+|w_1+\tilde w|^2(w_1+\tilde w) \, .$$


Since $\mathcal{L}_q$ is self-adjoint with respect to $\la \cdot, \cdot \ra$,
$$\partial_t \la \mathcal{L}_q w_1, w_1 \ra = 2\la \mathcal{L}_q w_1, \partial_t w_1 \ra = 2\la \mathcal{L}_q w_1, i\mathcal{L}_q w_1 \ra=0 \,.$$
By Lemma \ref{L:spec_lower_bd},
$$\|w_1(t)\|_{H^1}^2 \lesssim \la \mathcal{L}_q w_1(t), w_1(t) \ra = \la \mathcal{L}_q w_0, w_0 \ra \lesssim \|w_0\|_{H^1}^2 \, .$$
Hence, there exists a constant $c_1>0$ such that
\begin{equation}
\label{E:w1bd}
\|w_1(t)\|_{H^1}\leq c_1h \quad \text{for all }t \,.
\end{equation}
It can be checked by direct computation using $\mathcal{L}_q v_4=iv_3$ that
\begin{equation}
\label{E:PLcommute}
P\circ \mathcal{L}_q = \mathcal{L}_q \circ P = \omega(w,v_3)v_3 \,.
\end{equation}
(An abstract argument using the fact 
that $\mathcal{L}_q$ preserves $V_0$ can be given to justify the first equality, which is, in fact, all we use for now).  Let 
$$\hat w = \tilde w - P \tilde w \,$$
so that $P\hat w(t) =0$ for all $t$ (and hence Lemma \ref{L:spec_lower_bd} will be applicable to $f=\hat w(t)$).  Using \eqref{E:PLcommute}, we find that
$$\partial_t \hat w = -\tfrac12 i\mathcal{L}_q \hat w +(I-P)iF \,.$$
Using the self-adjointness of $\mathcal{L}_q$ with respect to $\la \cdot, \cdot \ra$ and the above equation for $\partial_t \hat w$, we find that
\begin{align*}
\partial_t \la \mathcal{L}_q \hat w, \hat w \ra &=  \la \mathcal{L}_q \hat w, \partial_t \hat w \ra 
= 2\la \mathcal{L}_q \hat w , -\tfrac12i\mathcal{L}_q\hat w + (1-P)F \ra\\
&= 2\la \mathcal{L}_q \hat w , F\ra  \,.
\end{align*}
Let $[0,T]$ be a time interval over which $\tilde w$ remains 
\begin{equation}
\label{E:tilde-w-bootstrap}
\|\tilde w\|_{L_{[0,T]}^\infty H_x^1} \leq c_1 h \,,
\end{equation}
where $c_1$ is given in \eqref{E:w1bd} (so that we know \emph{a priori}
 that $\tilde w$ is at least no worse that 
$w_1$, although of course we want to show that it is better).  [All future instances of $\lesssim$ mean ``less than a constant times'', where the constant depends upon $c_0$ (in Lemma \ref{L:spec_lower_bd}) and $c_1$.] 
This gives an estimate for $F$:  
$\|F\|_{L_{[0,T]}^\infty H_x^1} \lesssim h^2$.  
By integrating, for $0\leq t\leq T$, we have
$$| \la \mathcal{L}_q \hat w(t), \hat w(t) \ra | \lesssim 
 T \, \|\hat w\|_{L_{[0,T]}^\infty H_x^1} h^2 \,.$$
Taking the sup over $t\in [0,T]$ and employing Lemma \ref{L:spec_lower_bd},
$$
\| \hat w \|_{L_{[0,T]}^\infty H_x^1}^2 \lesssim
 T h^2 \|\hat w\|_{L_{[0,T]}^\infty H_x^1} \,,$$
and hence
\begin{equation}
\label{E:w_hat_bd}
\| \hat w \|_{L_{[0,T]}^\infty H_x^1} \lesssim  T h^2 \,. 
\end{equation}

From this, we need to infer a bound on $\tilde w$.  We compute
\begin{align*}
  \partial_t \, \omega(\tilde w, v_3) &= \omega( -\tfrac12 i \mathcal{L}_q \tilde w + iF, v_3)
  = -\tfrac12 \la \mathcal{L}_q \tilde w, v_3 \ra + \la F, v_3 \ra \\
&= \la F, v_3 \ra \,,
\end{align*}
and thus, on $[0,T]$, we have the bound
$$|\partial_t \, \omega(\tilde w, v_3)| \lesssim  h^2 \,.$$
Integrating in time, we find that for all $t\in [0,T]$,
\begin{equation}
\label{E:tilde-w-v3}
|\omega(\tilde w(t), v_3)| \lesssim h^2 \, t \,.
\end{equation}

Now we perform a similar computation for $\omega(\tilde w, v_4)$.
\begin{align*}
\partial_t \, \omega(\tilde w, v_4) &= \omega( -\tfrac12 i \mathcal{L}_q \tilde w + iF, v_4) 
= -\tfrac12 \la \mathcal{L}_q \tilde w, v_4 \ra + \la F, v_4 \ra \\
&= -\tfrac12 \la  \tilde w, iv_3 \ra+ \la F, v_4 \ra 
= \tfrac12 \omega(  \tilde w, v_3) + \la F, v_4 \ra \,. 
\end{align*}
Appealing to \eqref{E:tilde-w-v3}, we obtain the bound
$$|\partial_t \, \omega(\tilde w, v_4 ) | \lesssim  h^2 \, t +  h^2 \,, $$
which, integrated in time, yields
\begin{equation}
\label{E:tilde-w-v4}
| \omega(\tilde w(t), v_4 ) | \lesssim t^2 h^2 +  t h^2 \,.
\end{equation}
The estimates \eqref{E:tilde-w-v3} and \eqref{E:tilde-w-v4} give
$$\|P\tilde w\|_{L_{[0,T]}^\infty H_x^1} \lesssim h^2 \, T^2 +  h^2 \, T \,.$$
Now, provided $T\lesssim h^{-1/2}$, we obtain
$$\|P\tilde w\|_{L_{[0,T]}^\infty H_x^1} \leq \tfrac14 c_1 h \,,$$
and thus
$$\|\tilde w\|_{L_{[0,T]}^\infty H_x^1} = \|\hat w\|_{L_{[0,T]}^\infty H_x^1} + \|P\tilde w\|_{L_{[0,T]}^\infty H_x^1} \leq \tfrac12 c_1 h \,,$$
for $h$ suitably small (in terms of the constants $c_0$ and $c_1$).  Thus we have that the bootstrap assumption \eqref{E:tilde-w-bootstrap} indeed remains valid over $[0,T]$, and moreover, that
$$\|\tilde w(t)\|_{H^1} \lesssim t(1+t)h^2 $$
holds over the whole interval $[0,T]$.
\end{proof}

The following corollary is useful in streamlining Theorem \ref{T:main}:

\begin{corollary}
\label{C:npt}
Suppose the hypothesis of Theorem \ref{T:npt} holds, and that $\tilde\lambda$ satisfies $|\lambda-\tilde \lambda| \lesssim h^2$.  Then, in place of \eqref{eq:t1}, we have
$$\|u(t)-e^{it\tilde \lambda^2/2}\big( v_{\tilde \lambda} + e^{-it\mathcal{L}_{\tilde \lambda,q}}w_0\big)\|_{H_x^1} \leq C(1+t)^2h^2\,.$$
\end{corollary}
\begin{proof}
It suffices to show that for any $f$ such that $P_{\lambda,q} f=0$, we have
$$\|e^{-it\mathcal{L}_{\lambda,q}}f - e^{-it\mathcal{L}_{\tilde \lambda,q}}f\|_{H_x^1} \lesssim h^2(1+t)\|f\|_{H_x^1} \,.$$
(In fact, this is stronger than necessary since $\|w_0\|_{H_x^1} \leq h$.  The dominant error term arises from the fact that $\|e^{it\lambda^2/2}v_\lambda - e^{it\tilde\lambda^2/2}v_{\tilde\lambda}\|_{H_x^1} \lesssim h^2$.)
Let $u(t) = e^{-it\mathcal{L}_{\lambda,q}}f$ and $\tilde u(t) = e^{-it\mathcal{L}_{\tilde \lambda,q}}f$.  Henceforth we will drop the $q$ subscript.  Then,
$$\partial_t (u-\tilde u) = -i\mathcal{L}_{\lambda}(u-\tilde u) +i(\mathcal{L}_\lambda-\mathcal{L}_{\tilde \lambda})\tilde u\,.$$
Note that
$$(\mathcal{L}_\lambda - \mathcal{L}_{\tilde \lambda})\tilde u = \tfrac12(\lambda^2-\tilde\lambda^2)\tilde u - 2(v_\lambda^2-v_{\tilde\lambda}^2)\tilde u -(v_\lambda^2-v_{\tilde\lambda}^2)\bar{\tilde u} \,.$$
As in the proof of Theorem \ref{T:npt}, we compute
$$\partial_t \la \mathcal{L}_\lambda (u-\tilde u), u-\tilde u\ra = 2\la \mathcal{L}_\lambda(u-\tilde u), \;\partial_t(u-\tilde u) \ra \,.$$
Substituting the above formulas and using that $\la \mathcal{L}g,i\mathcal{L}g\ra =0$, we obtain
$$\partial_t \la \mathcal{L}_\lambda (u-\tilde u), u-\tilde u\ra = 2\la \mathcal{L}_\lambda(u-\tilde u) , \; \big( \tfrac12(\lambda^2-\tilde\lambda^2)\tilde u - 2(v_\lambda^2-v_{\tilde \lambda}^2)\tilde u - (v_\lambda^2-v_{\tilde\lambda}^2)\bar{\tilde u} \big) \ra \,.$$
The next step is to write out the operator $\mathcal{L}_\lambda$, and address the above expression term by term.  For the $\partial_x^2$ term, integrate by parts once, and then apply the Cauchy-Schwarz inequality.  For the $\delta_0$ term, use that $|g(0)| \leq \|g\|_{H_x^1}$.  For all other terms, directly apply 
the Cauchy-Schwarz inequality.  The resulting bound is
$$|\partial_t \la \mathcal{L}_\lambda(u-\tilde u), u-\tilde u \ra| \lesssim |\lambda-\tilde\lambda| \|u-\tilde u\|_{H_x^1}\|\tilde u\|_{L_{[0,t]}^\infty H_x^1} \,.$$
Following the argument used to obtain the bound \eqref{E:w1bd} in the proof of Theorem \ref{T:npt}, we conclude that
$$\|u(t)\|_{H_x^1} \lesssim \|f\|_{H_x^1} \,,$$
uniformly for all $t$.  Using that $\|\tilde u\|_{H_x^1} \leq \|u-\tilde u\|_{H_x^1} + \|\tilde u\|_{H_x^1}$, we obtain
$$|\partial_t \la \mathcal{L}_\lambda(u-\tilde u), u-\tilde u \ra| \lesssim h^2 \|u-\tilde u\|_{H_x^1}^2 + h^2\|u-\tilde u\|_{H_x^1}\|f\|_{H_x^1} \,.$$
Integrating over $[0,t]$ and using that $u(0)=\tilde u(0)$, we obtain
\begin{equation}
\label{E:coer3}
 |\la \mathcal{L}_\lambda(u(t)-\tilde u(t)), u(t)-\tilde u(t) \ra| \lesssim th^2 \|u-\tilde u\|_{L_{[0,t]}^\infty H_x^1}^2 + th^2 \|u-\tilde u\|_{L_{[0,t]}^\infty H_x^1}\|f\|_{H_x^1} \,.
\end{equation}
By Lemma \ref{L:spec_lower_bd}, 
$$\|(u-\tilde u) - P_\lambda (u-\tilde u) \|_{H_x^1}^2 \lesssim \la \mathcal{L}_\lambda\big( (u-\tilde u) - P_\lambda (u-\tilde u) \big), \; \big( (u-\tilde u) - P_\lambda (u-\tilde u)\big) \ra \,.$$
By Cauchy-Schwarz, we deduce the bound
\begin{equation}
\label{E:coer2}
\|u-\tilde u \|_{H_x^1}^2 \lesssim \la \mathcal{L}_\lambda(u-\tilde u), u-\tilde u \ra + (\|u-\tilde u\|_{H_x^1} + \|P_\lambda(u-\tilde u)\|_{H_x^1})\|P_\lambda(u-\tilde u)\|_{H_x^1} \,.
\end{equation}
Note that
\begin{equation}
\label{E:p-lambda}
P_\lambda(u-\tilde u) = - P_\lambda\tilde u =-(P_\lambda-P_{\tilde \lambda})\tilde u -P_{\tilde \lambda} \tilde u\,,
\end{equation}
Since $P_{\tilde \lambda} \circ e^{-it\mathcal{L}_{\tilde \lambda}} = e^{-it\mathcal{L}_{\tilde \lambda}}\circ P_{\tilde \lambda}$,
$$ P_{\tilde \lambda}\tilde u = e^{-it\mathcal{L}_{\tilde \lambda}} P_{\tilde \lambda} f .$$
This is just the evolution of the (generalized) kernel, and hence
$$\| P_{\tilde \lambda}\tilde u\|_{H_x^1} \leq t\|P_{\tilde \lambda} f\|_{H_x^1} = t\|(P_{\tilde \lambda} - P_\lambda)f\|_{H_x^1} \lesssim th^2\|f\|_{H_x^1}\,.$$
Similarly,
$$\|(P_\lambda-P_{\tilde \lambda})\tilde u\|_{H_x^1} \lesssim h^2 \|\tilde u\|_{H_x^1} \lesssim h^2( \|u-\tilde u\|_{H_x^1} + \|f\|_{H_x^1}) \,,$$
which, together with the previous bound, gives (see \eqref{E:p-lambda})
$$\|P_\lambda(u-\tilde u) \|_{H_x^1} \lesssim h^2\|u-\tilde u\|_{H_x^1} + (1+t)h^2\|f\|_{H_x^1}^2 \,.$$
Combining this with \eqref{E:coer2}, we obtain the bound (absorbing $ h^2\|u-\tilde u\|_{H_x^1}^2$ on the left side)
$$\|u-\tilde u \|_{H_x^1}^2 \lesssim \la \mathcal{L}_\lambda (u-\tilde u), u-\tilde u \ra + (1+t)h^2\|u-\tilde u\|_{H_x^1}\|f\|_{H_x^1} + (1+t)^2h^4\|f\|_{H_x^1}^2 \,.$$
Combining this with \eqref{E:coer3}, provided $th^2\ll 1$, we obtain the bound
$$\|u-\tilde u\|_{L_{[0,t]}^\infty H_x^1}^2 \lesssim (1+t)h^2\|u-\tilde u\|_{H_x^1}\|f\|_{H_x^1} + (1+t)^2h^4\|f\|_{H_x^1}^2 \,,$$
from which it follows that
$$\|u(t)-\tilde u(t)\|_{H_x^1} \lesssim (1+t)h^2 \|f\|_{H_x^1} \,,$$
which implies the estimate in the corollary.
\end{proof}

\section{The free case}
\label{fc}

We will discuss the case of $ q = 0 $ and the initial data close
to a stationary soliton $ \sech x $. Very precise information 
can in principle be obtained in this case using 
the inverse scattering method 
\cite{ZS72} -- see also \cite{DZ}, \cite{DIZ},\cite{FT}.
However we are not aware 
of any reference containing that information -- see \cite[Appendix B]{HMZ1}
for a discussion.

\subsection{Spectral theory of the linearized operator}
\label{sa}
The explicit spectral decomposition of the operator $ H_0 $ was 
discovered by Kaup \cite{Kaup76} (see also \cite{Ya} for a recent 
discussion and generalizations). We now
present it in a way which will make the spectral 
decomposition of $ H_q $  natural.

Spectral theory of operators of the form
\begin{equation}
\label{eq:form}  H = \frac 12 \begin{bmatrix} - \partial^2_x + 1 & \ \ \ 0 \\
\ \ \ 0 & \partial_x^2 - 1 \end{bmatrix} + 
\begin{bmatrix} \ V_1 & \ V_2 \\ -V_2 & -V_1 
\end{bmatrix} \,, 
\end{equation}
was studied systematically by Buslaev-Perelman \cite{BP} 
and Krieger-Schlag \cite{KS}. Despite the 
non-normality of $ H $ (if $ V_2 \neq 0 $) a spectral 
decomposition is available once the existence and properties of 
the four dimensional set of solutions to 
\[  H \psi  = ( k^2 + 1 ) \psi \]
is established. 
They are characterized by their behaviour as $ x \rightarrow \infty $:
\[ \psi ( x ) \sim e^{ \pm i k x } \begin{bmatrix} 1 \\ 0 \end{bmatrix}\,, \ \
 e^{ \pm \mu x } \begin{bmatrix} 0 \\ 1 \end{bmatrix}\,, \ \ 
\mu^2 = k^2 + 2 \,, \ \ k \neq 0 \,, \ \ \mu \neq 0 \,, \]
see \cite[Section  5]{KS} for a careful discussion. 

For the linearization of the cubic NLS the set of four solutions can be found 
{\em explicitly}\footnote{In Appendix \ref{ndkp} we show how this solution can be guessed by performing 
a simple numerical experiment even if, as we were, one is ignorant of 
the inverse scattering developments. We are grateful to Galina Perelman for
explaining to us the structure of solutions to linearized operators
in the completely integrable case.}
 and it is given by 
\begin{equation}
\label{eq:Kaup4} \{ \, \Psi_+(\cdot, k), \quad \Psi_+(\cdot, -k), \quad 
\Psi_-(\cdot, i\mu), \quad \Psi_-(\cdot, -i\mu) \, \}\,, \ \ 
k \neq 0 \,, \ \ \mu \neq 0 \,, 
\end{equation}
where
\begin{equation}
\label{eq:KaupP}
 \Psi_+ ( x , k ) =  \begin{bmatrix}
( \tanh x - i k )^2 \\ - \sech^2 x \end{bmatrix} e^{ ik x } \, , \quad
\Psi_- = \sigma_1 \Psi_+ \, ,\quad 
\mu = (k^2+2)^{1/2} \, .
\end{equation}

Since $\sigma_1 H_0 \sigma_1 = -H_0$, we have that
$$\{ \, \Psi_-(\cdot, k), \quad \Psi_-(\cdot, -k), \quad 
\Psi_+(\cdot, i\mu), \quad \Psi_+(\cdot, -i\mu) \, \}$$
is a basis for the solution space to $H_0\psi = -(k^2+1)\psi$.

Now let
\[  \varphi_+( x , k )  = \sigma_3 {\Psi_+ ( x , k ) }  \,, \ \
\varphi_- ( x, k )  = \sigma_1 \varphi_- ( x , k ) \,, \]
Then since $\sigma_3 H_0\sigma_3 = H_0^*$, we have that
$$\{ \, \varphi_+(\cdot, k), \quad \varphi_+(\cdot, -k), \quad \varphi_-(\cdot, i\mu), 
\quad \varphi_+(\cdot, -i\mu) \, \}$$
is a basis of the solution space to $H_0^*\psi = (k^2+1)\psi$.  Note that
$$\varphi_- = \sigma_1 \varphi_+ = \sigma_1\sigma_3 \Psi_+ = -\sigma_3\sigma_1\Psi_+ 
= -\sigma_3\Psi_-$$
Finally, using that $\sigma_1 H_0^* \sigma_1 = -H_0^*$, we obtain that 
$$\{ \, \varphi_-(\cdot, k), \quad \varphi_-(\cdot, -k), \quad \varphi_+(\cdot, i\mu), 
\quad \varphi_+(\cdot, -i\mu) \, \}$$
is a basis of the solution space to $H_0^*\psi = -(k^2+1)\psi$ for $k \neq 0$.

The eigenvalue $0$ corresponds to a generalized eigenspace 
$\operatorname{span}\{ \Psi_1, \Psi_2, \Psi_3, \Psi_4 \}$, to be described 
now.  Let $\eta(x)=\sech x$ and
$$e_1 = -\partial \qquad e_2 = ix \qquad e_3 = i \qquad e_4 =\partial x$$
Then $e_1$ (translation) and $e_2$ (Galilean) are symplectically dual and 
we have \eqref{E:kernel1}.
Let
$$\Psi_j = i^{j-1}U(e_j\cdot \eta)\,, $$
where $ U $ is given in \eqref{eq:Hql}.
Then
$$H_0\Psi_1 = H_0\Psi_3 = 0, \qquad H_0\Psi_2=\Psi_1, \quad H_0\Psi_4=\Psi_3$$
The generalized kernel of $ H_0^* $ is spanned by $ \varphi_j \defeq 
\sigma_3 
\overline 
\Psi_j $,
\[ H_0^* \varphi_1 = H_0^* \varphi_3 = 0 \,, \ H_0^* \varphi_2 = \varphi_1 \,, \ H_0^* \varphi_4 = 
 \varphi_3 \,.\]
Since $ U^* \sigma_3 U = \sigma_2 $ we also see that 
\[ \begin{split}
& \int_\RR \varphi_2 (x )^* \Psi_1 ( x ) dx = \int_\RR \varphi_1^* ( x ) \Psi_2 (x ) 
= 1 \,, \\
& \int_\RR \varphi_4 (x )^* \Psi_3 ( x ) dx = \int_\RR \varphi_3^* ( x ) \Psi_4 (x ) 
= -1 \,.
\end{split} \]

Finally we recall that
$H_0$ has a simple threshold resonance given by the explicit formula 
(see Chang-Gustafson-Nakanishi-Tsai \cite[\S 3.7]{CGNT})
\begin{equation}
\label{eq:ress}
\begin{bmatrix} \tanh^2 x \\ - \sech^2 x \end{bmatrix} \,,
\end{equation}
corresponding to $ k = 0 $.  
Following \cite{BP} and \cite[Definition 5.18]{KS} (note a slight change of
convention between this paper and \cite{KS}) we say
that $ H $ has a {\em resonance} at $ 1 $ if there
exists $ u \in L^\infty $ such that $ Hu = u $. The multiplicity
of a resonance is the number of independent solutions with these properties.
As we will recall below the maximum multiplicity is $ 2$. Here we include
eigenvalues as resonances: ``true'' resonances satisfy $ u \in L^\infty
\setminus L^2 $.

This definition is equivalent to the more general definition based on the
meromorphic continuation of the resolvent. The potential $ \sech^2 x $
is exponentially decaying and the resolvent of $ H_0 $ (the same operator
without the potential term),
$ R_0 ( z ) = ( H_0 - z )^{-1} $,
has a global meromorphic continuation to a
three sheeted Riemann surface, with poles at $ \pm 1 $. The resolvent
$ R ( z ) = ( H - z )^{-1} $ can then be continued from the
{\em physical plane},
$ \Sigma \defeq \CC \setminus (( - \infty , -1 ) \cup \{ 0 \}
\cup ( 1 , \infty ) ) $, to a neighbourhood of $ \Sigma $ on that
three sheeted Riemann surface. Near $ \pm 1 $ the resolvent is
meromorphic in $ \lambda$, $ z = \pm (1 + \lambda^2 )$.
The analysis outlined in \cite{BP} (see \cite[Lemma 5.2]{KS} and
\cite[Lemma 6.5]{KS} for detailed presentation)
can be used to show that the definition of
resonances as poles of the resolvent coincides with the definition above
given in terms of solutions.

Let $P_c$ denote the symplectic orthogonal projection onto the essential 
spectrum, which we define as $I-P_d$ where $P_d$ is the symplectic orthogonal 
projection onto the discrete spectral subspace $E_0$.  The $2\times 2$ matrix 
kernel $P_c(x,y)$ of $P_c$ is given by Kaup's formula \cite{Kaup76}
$$P_c(x,y) = \frac1{2\pi} \int_\mathbb{R}  \frac{1}{(1+k^2)^2} 
(\Psi_+(x,k)\varphi_+(y,k)^* + \Psi_-(x,k)\varphi_-(y,k)^* )\, dk$$
Once we know \eqref{eq:Kaup4}, this formula can also be derived by contour 
deformation and the fact that
$$\frac{1}{2\pi i} \int_{\Gamma} (H_0-z)^{-1} \, dz = \text{Id} \,$$
where $\Gamma$ is any contour that encloses the spectrum of $H_0$ -- see 
\cite[Lemma 6.8]{KS}. As claimed in \cite{Kaup} and \cite{Ya} it can 
also be checked by an explicit calculations of the integral.

We now put this into a form that is more consistent with one-dimensional 
scattering theory (see for instance \cite{TZ}) and connects the basis
with the basis of {\em scattering solutions} of \cite[\S 2.5.1]{BP} and
\cite[\S 6]{KS}. 

Let 
\begin{equation}
\label{E:v+} 
v_+(x,k) = \frac{1}{(1+i|k|)^2}\Psi_+(x,k) \,. \end{equation}
Then $H_0 v_+ = (1+k^2)v_+$ and
\begin{equation}
\label{E:v++}
v_+(x,k) \sim \begin{bmatrix} 1 \\ 0 \end{bmatrix} 
\begin{cases} e^{ikx} + R_+(k)e^{-ikx} & \text{as }x \to -\infty  \\   
T_+(k)e^{ikx} & \text{as }x\to +\infty \end{cases} \, \quad \text{for }k>0
\end{equation}
\begin{equation}
\label{E:v+-}
v_+(x,k) \sim \begin{bmatrix} 1 \\ 0 \end{bmatrix} 
\begin{cases} T_-(k)e^{ixk} & \text{as }x \to -\infty  \\  
e^{ikx}+R_-(k)e^{-ikx} & \text{as }x\to +\infty \end{cases} \, 
\quad \text{for }k<0
\end{equation}
with 
$$R_+(k)=0 \quad \text{and} \quad T_+(k)= \frac{(1-ik)^2}{(1+ik)^2}$$
$$R_-(k)=0 \quad \text{and} \quad T_-(k)=  \frac{(1+ik)^2}{(1-ik)^2}$$

Now let 
\begin{align*}
v_- \defeq \sigma_1 v_+  &\implies H_0v_- = -(1+k^2)v_+ \\
\tilde v_+ \defeq \sigma_3 v_+ &\implies H_0^* \tilde v_+ = (1+k^2)\tilde v_+ \\
\tilde v_- \defeq \sigma_1 \tilde v_+ &\implies H_0^* \tilde v_- 
= -(1+k^2)\tilde v_-
\end{align*}

Then
\begin{equation}
\label{eq:Pcv} P_c(x,y) = \frac{1}{2\pi} \int_\RR
(v_+(x,k)\tilde v_+(y,k)^* + v_-(x,k)\tilde v_-(y,k)^* ) \, dk\,, 
\end{equation}
and
\begin{equation}
\label{E:orth}
\frac1{2\pi} \int_\mathbb{R} \tilde v_\pm(x,k)^* v_\pm(x,k') \, dx 
= \delta(k-k'), \quad \frac1{2\pi} 
\int_\mathbb{R} \tilde v_\pm(x,k)^* v_\mp(x,k') \, dx =0
\end{equation}

\subsection{The free linearized propagator.}
\label{S:fprop}
It follows from \eqref{eq:Pcv} that 
the propagator $e^{-\frac12it H_0}P_c$ on the essential 
spectrum is represented by the Schwartz kernel 
$$e^{-\frac12it H_0}P_c(x,y) = \frac1{2\pi} \int_\mathbb{R} 
( e^{-\frac12 it(1+k^2)} v_+(x,k)\tilde v_+(y,k)^* 
+ e^{
\frac12 it(1+k^2)} v_-(x,k)\tilde v_-(y,k)^* )  dk \,.  $$
We will now study $ e^{-\frac12it H_0}P_c w_0 $ for $ w_0 $ 
appearing in Theorem \ref{T:npt}.

\begin{proposition}
\label{p:w0}
Suppose that $ w_0 \in {\mathcal S} ( \RR ) $ is real valued. Then 
\begin{gather}
\label{eq:lw0}
\begin{gathered}
e^{-\frac 12  i t H_0 } 
P_c w_0(x,t) = \frac1{\sqrt{2\pi}} \int_{-\infty}^{+\infty} (a(x,k) e^{-\frac12 i(k^2+1)t} + b(x,k) e^{\frac12 i(k^2+1)t})f(k) dk \,, \\  
f(k) = \frac{1}{\sqrt{2\pi} (1-i|k|)^2}\int_{-\infty}^{+\infty} (1+2ikt(x)-k^2t(x)^2)w_0(x) e^{-ikx} \, dx
\\
a(x,k)= \frac{ (t(x)-ik)^2}{(1+i|k|)^2}e^{ikx} \,, \ \ \ 
b(x,k)= \frac{-s(x)^2}{(1+i|k|)^2} e^{-ikx} \,, 
\end{gathered}
\end{gather}
where we used the notation 
$ t ( x )  = \tanh x $ and 
$ s ( x ) = \sech x $. 
Consequently,
\begin{equation}
\label{eq:lw01}
e^{-\frac12 i t H_0 } 
P_c w_0(0,t) =  -\frac{1}{\sqrt{ 2 \pi  t} } e^{\frac12it}e^{i\frac{\pi}{4}} \int_\RR w_0 ( x ) dx 
 + {\mathcal O} ( t^{-3/2})  \,. \end{equation}
\end{proposition}
\begin{proof}
Let 
\begin{align*}
&V_+g(x) = \frac1{\sqrt{2\pi}} \int_k v_+(x,k) g(k) \, dk \,,
\end{align*}
be the ``inverse distorted Fourier transform,'' which gives 
\begin{align*}
&V_+^*f(k)=\frac1{\sqrt{2\pi}}\int_x v_+^*(x,k)f(x) \, dx \,,
\end{align*}
the ``distorted Fourier transform,'' associated to the operator $H_0$.  
With this notation, we have
$$P_ce^{-\frac12itH} = V_+M(t)V_+^*\sigma_3 + \sigma_1 V_+M(-t)V_+^*\sigma_3\sigma_1, \quad M(t)f(k) \defeq e^{-\frac12i(k^2+1)t}f(k)$$
Consequently, for $w_0$ real, 
\begin{align*}
e^{-\frac12 i t H_0 } P_c w_0 & 
= \begin{bmatrix} 1 & 0 \end{bmatrix}
 \;  P_c e^{-\frac12 itH} \begin{bmatrix} 1 \\ 1 \end{bmatrix} w_0 \\
&= \begin{bmatrix} 1 & 0 \end{bmatrix} 
 \; (V_+M(t)V_+^* \sigma_3 + \sigma_1 V_+ M(-t)V_+^*\sigma_3\sigma_1) \begin{bmatrix} 1 \\ 1\end{bmatrix} w_0 \\
&= \begin{bmatrix} 1 & 0 \end{bmatrix} \; (V_+ M(t) + \sigma_1 V_+M(-t)) V_+^* \begin{bmatrix} 1 \\ -1\end{bmatrix} w_0
\end{align*}
We write the function $v_+(x,k)$ given by \eqref{E:v+} 
as
$$v_+(x,k) = \begin{bmatrix} a(x,k) \\ b(x,k) \end{bmatrix} e^{ikx} \; , \qquad \begin{aligned} 
a(x,k)= \frac{ (t(x)-ik)^2}{(1+i|k|)^2}e^{ikx} \\
b(x,k)= \frac{-s(x)^2}{(1+i|k|)^2} e^{-ikx}
\end{aligned}
$$
and define
\begin{equation}
\label{eq:deff}
f(k) \defeq \left(V_+^* \begin{bmatrix} 1 \\ -1 \end{bmatrix} w_0 \right)(k) 
= \frac1{\sqrt{2\pi}} \int_x (\bar a(x,k) - \bar b(x,k)) w_0(x) \, dx \,. 
\end{equation}
Then, 
\begin{align*}
e^{-\frac12 it H} P_c w(0,t) &= \frac1{\sqrt{2\pi}} \int (a(0,k)e^{-\frac12 i(k^2+1)t} + b(0,k)e^{\frac12i(k^2+1)t})f(k) \, dk \\
&= \frac{1}{\sqrt{2\pi}} \int \frac{-k^2e^{-\frac12i(k^2+1)}-e^{\frac12i(k^2+1)t}}{(1+i|k|)^2} \, f(k) \, dk\\
& =  -\frac{1}{\sqrt {t} } e^{\frac12it}e^{i\frac{\pi}{4}}f(0)
 + {\mathcal O} ( t^{-3/2})  \,, 
\end{align*}
by the method of stationary phase.
\end{proof}

\begin{figure}
\begin{center}
\includegraphics[width=6in]{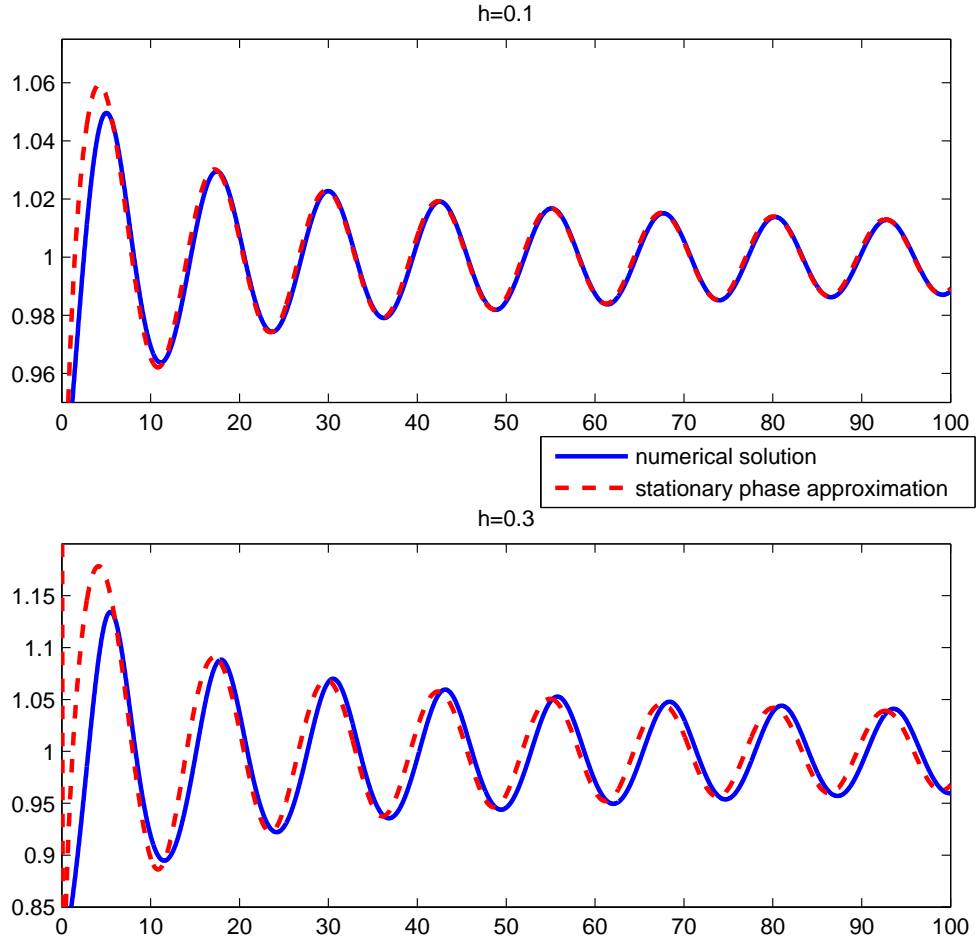}
\end{center}
\caption{Breathing patterns for $ q = 0 $ and
$ u( x , 0 ) = ( 1 + h )/(1+2h) \sech (x/(1+2h)  $ (the rescaling 
is rigged so that the resulting ground state is simply $ \sech x $): we 
show $ |u ( 0 , t )| $ and the asymptotic 
prediction 
for $ h = 0.1 $, $ h = 0.3 $. The agreement is
remarkably good for times much longer than given in the theoretical 
result. A more precise statement can be obtained using the inverse
scattering method.
}
\label{f:2}
\end{figure}

\subsection{Nonlinear perturbation theory in the free case}
\label{nlpfree}
Let us take a particular example:  
$$w_0(x) = \frac{1+h}{1+2h}\sech\left( \frac{x}{1+2h} \right) - \sech x $$
(the choice of scaling was made so that $ \sech x $ is selected as the 
nonlinear ground state by Theorem \ref{T:npt}).
In this case we compute
$ f(0) = h \sqrt {{\pi}/{2}} $
which gives
$$e^{-\frac 12 i t  H_0 } P_c w_0(0,t) = 
- \sqrt{\frac{\pi}{2}} \, e^{\frac14i\pi} \, \frac{1}{\sqrt{t}} \; e^{\frac12it}h 
+ {\mathcal O} \left( \frac h { t^{3/2} } \right) $$
We also have $P_3 w_0 = \omega(w_0,v_4)v_3 = 0$ and $P_4 w_0 = \omega(w_0,v_3)v_4 \approx 0.4\cdot h^2 v_4$.  Thus,  for $t \gg 1$,  we have
$$ e^{-\frac 12 i t  H_0 } P_c w_0(0,t) =  0.2\cdot i h^2  t - 
 \, e^{\frac14i\pi} \, \sqrt{\frac{\pi}{2t}} \; e^{\frac12it}h + {\mathcal O} (h^2) + {\mathcal O} 
\left( \frac h { t^{3/2} } \right)  $$
We can now apply Theorem \ref{T:npt} to see that the solution of 
\[ i u_t = - u_{xx}/2 - |u|^2 u \,, \ \ u ( x , 0 ) = 
 \frac{1+h}{1+2h}\sech\left( \frac{x}{1+2h} \right) \,, \]
satisfies
\begin{equation}
\label{eq:u0tf}
e^{-i t/2}  u ( 0 , t ) = 
1 - e^{\frac14i\pi} \, \sqrt{\frac{\pi}{2t}} \; e^{\frac12it}h + {\mathcal O} (t^2 h^2) 
+ {\mathcal O} 
\left( \frac h { t^{3/2} } \right) \,, \ \ 1 \ll t \ll h^{-1/2} \,. 
\end{equation}
Figure \ref{f:2} compares this asymptotic expression with the numerical solution.

\medskip

\noindent
{\bf Remark.}
We should stress that a more precise result valid for all values of $ h $ 
can in principle be obtained using the inverse scattering method -- see
\cite[Appendix B]{HMZ1} and references given there. It would be very interesting
to compare those exact expressions with our rough asymptotics. The results
of \cite[Appendix B]{HMZ1} show already that \eqref{eq:u0tf} can be corrected
since we know that
\[    u ( x , t ) = e^{i \varphi ( h )} \sech x + {\mathcal O}_{L^\infty } \left( 
\frac 1 { \sqrt{t} } \right) \,, \]
where 
\[ \begin{split}  \varphi  ( h )  & =  \int_0^\infty 
\log\left( 1 + \frac{\sin^2\pi h  }
{\cosh^2\pi \zeta} \right) \frac{\zeta}{\zeta^2+(1 + 2 h )^2} \, d\zeta 
\\ & 
\simeq \pi^2 h^2 \int_0^\infty \frac{\zeta \sech^2 \pi \zeta} { 1 + \zeta^2 } d\zeta 
\simeq 0.6 \; h^2 
\,, \ \ h \longrightarrow 0 
\end{split}\]
Hence, in the application of Theorem \ref{T:npt} the error terms $ {\mathcal O} 
( h^2) $ in 
\eqref{eq:u0tf} are optimal.

\section{Small  external delta potential}
\label{sedp}

In this section we will use Theorem \ref{T:npt} to prove Theorem \ref{T:main} stated
in the introduction. For that we will follow the same path as in \S \ref{fc}
and provide a spectral decomposition of the linearized operator with the 
$ \delta_0 $ potential. The scattering coefficients, $ R_\pm $ and $ T $, 
appearing in \eqref{E:v++} and \eqref{E:v+-} are now more singular which 
makes the asymptotic analysis more complicated.

\subsection{Basis of solutions to $ H_q \psi = \pm ( k^2 + 1 ) \psi $.}
\label{S:dfb}
Using the Kaup basis \eqref{eq:Kaup4} for the free problem we find 
a complete set of solutions $\psi$ to the equation 
$H_q \psi = (k^2+1)\psi$, where
$$H_q = 
\begin{bmatrix}
-\partial_x^2 + 1  & 0\\
0 & \partial_x^2 - 1
\end{bmatrix}
+2\sech^2(x+\sgn(x)\theta)\begin{bmatrix}
-2 & -1\\
1 & 2
\end{bmatrix}
-2q
\begin{bmatrix}
\delta_0 & 0 \\
0 & -\delta_0
\end{bmatrix} \,, 
$$ 
with $\theta = \tanh^{-1}q$, see \eqref{eq:Hql}.

Let  $s=\sech(x+\sgn(x)\theta)$, $t=\tanh(x+\sgn(x)\theta)$, and $\mu = (k^2+2)^{1/2}>0$.  With unknown coefficients $A(k)$, $B(k)$, $C(k)$, and $D(k)$, we look for $\psi(x,k)$ of the form
\begin{equation}
\label{eq:psi0}
\psi =
\begin{aligned}[t]
& \left( 
\begin{bmatrix}
(t-ik)^2 \\ -s^2 
\end{bmatrix} e^{ik(x-\theta)}
+
A\begin{bmatrix}
-s^2 \\ (t-\mu)^2
\end{bmatrix} e^{\mu (x-\theta)}
\right) x_-^0 \\ 
& +\left(
B\begin{bmatrix}
(t-ik)^2 \\ -s^2
\end{bmatrix} e^{ik(x+\theta)}
+C\begin{bmatrix}
(t+ik)^2 \\ -s^2 
\end{bmatrix} e^{-ik(x+\theta)}
+D\begin{bmatrix}
-s^2 \\ (t+\mu)^2
\end{bmatrix} e^{-\mu (x+\theta)}
\right) x_+^0
\end{aligned}
\end{equation}

For the unknowns $A(k)$, $B(k)$, $C(k)$, and $D(k)$, two equations are obtained by requiring continuity at $x=0$ and two more equations are obtained by requiring the appropriate jump condition in the derivatives at $x=0$.  This gives rise to the 4$\times$4 system analysed in detail in Appendix \ref{AA}.  

By comparing $\psi(x,-k)$ and $\overline{\psi(x,k)}$ asymptotically as $x\to -\infty$, and noting that both solve $H_q\psi=(1+k^2)\psi$, we find that $\psi(x,-k)=\overline{\psi(x,k)}$ and hence $A(-k)=\overline{A(k)}$, and similarly for $B$, $C$, and $D$.

Here is a typical consequence of the formulas from Appendix \ref{AA}.  An eigenvalue at $ 1 + k^2 $
comes from finding a solution to 
\[  B (k, q ) =0 \,, \]
with $ \Im k < 0 $.  More generally, a solution will give a resonance or
a pole of the resolvent. The following lemma is derived from the computations in Appendix \ref{AA}.

\begin{lemma}
\label{l:reig}
For $q<0$, $ 0 < |q| \ll 1 $ the operator $ H_q $
has 
one eigenvalue, $ \mu^\pm_q $,  near $ \pm 1 $,
\begin{equation}
\label{eq:reig+}
\mu_q^\pm =  1 - q^2 \,. 
\end{equation}
The corresponding eigenfuctions $ u_q^\pm $ can be chosen to be real and
satisfy $ \sigma_1 u_q^\pm = u_q^{\mp} $, where  
$$
u_q^+(x) = |q|^{1/2}\begin{bmatrix}
(\tanh(|x|+\theta)-q)^2 \\
-\sech^2(|x|+\theta)
\end{bmatrix}
e^{q(|x|+\theta)} \,, \ \ \theta = \tanh^{-1} q \,. 
$$
Consequently,
\begin{equation}
\label{eq:reig1}
  | u_q^\pm ( x ) | \leq C |q|^{\frac12} e^{-|q x|} \,, \ \   \int u_q^\pm ( x)^* 
\, \sigma_3 u_q^\pm (x ) \, dx = 1 \,.
\end{equation}
For $ 0< q \ll 1 $ the operator $ H_q $ has no eigenfunctions
near $ \pm 1 $ and the thresholds $ \pm 1 $ are {\em not} resonances.
\end{lemma}

\noindent
{\bf Remark.} The normalization of $ u_q^\pm $ is consistent with the 
spectral decomposition of $ H_q $ -- see \S \ref{ptm}. 

\medskip

We next analyse what happens when $ k = - i $ and $ \mu = 1 $. 
Since explicit formul{\ae} in that 
case do not play a r\^ole in our analysis, the spectrum of $ H_q $ (and equivalently 
of $ F_q = -i {\mathcal L}_q $) near
zero is analyzed by more general methods in Appendix \ref{S:pnz}. There
we proof the following lemma:

\begin{lemma}
\label{l:B}
For $ 0 < |q| \ll 1 $ the generalized 
kernel of $ F_q $ is given by $ \{ i v_1  , \partial_\lambda v_\lambda |_{\lambda = 1 } 
 \} $, 
\[  F_q ( i v_1  ) = 0 \,, \ \  F_q (  \partial_\lambda v_\lambda |_{\lambda = 1 } ) =
i v_1  \,. \]
In a neighbourhood of $ 0 $, $ F_q $ has two eigenvalues
\[ \lambda^\pm_q = \left\{ \begin{array}{ll} \pm  q^\frac12 + {\mathcal 
O} ( q^{3/2} )  & q >  0 \,, \\
 \pm  i |q|^\frac12 + {\mathcal O} ( |q|^{3/2} )  & q < 0 \,. \end{array}
\right.
\]
The two eigenfuctions, $ w_q^\pm $, are {\em odd}, and satisfy
$ \sigma_3 w_{q}^{\pm } = w_{q}^{\mp} $.
\end{lemma}

\begin{remark}
Note that $H_q=2F_q$, and thus the eigenvalues of $H_q$ occur at $2\lambda_q^\pm$.
\end{remark}

\begin{figure}
\begin{center}
\includegraphics[width=6in]{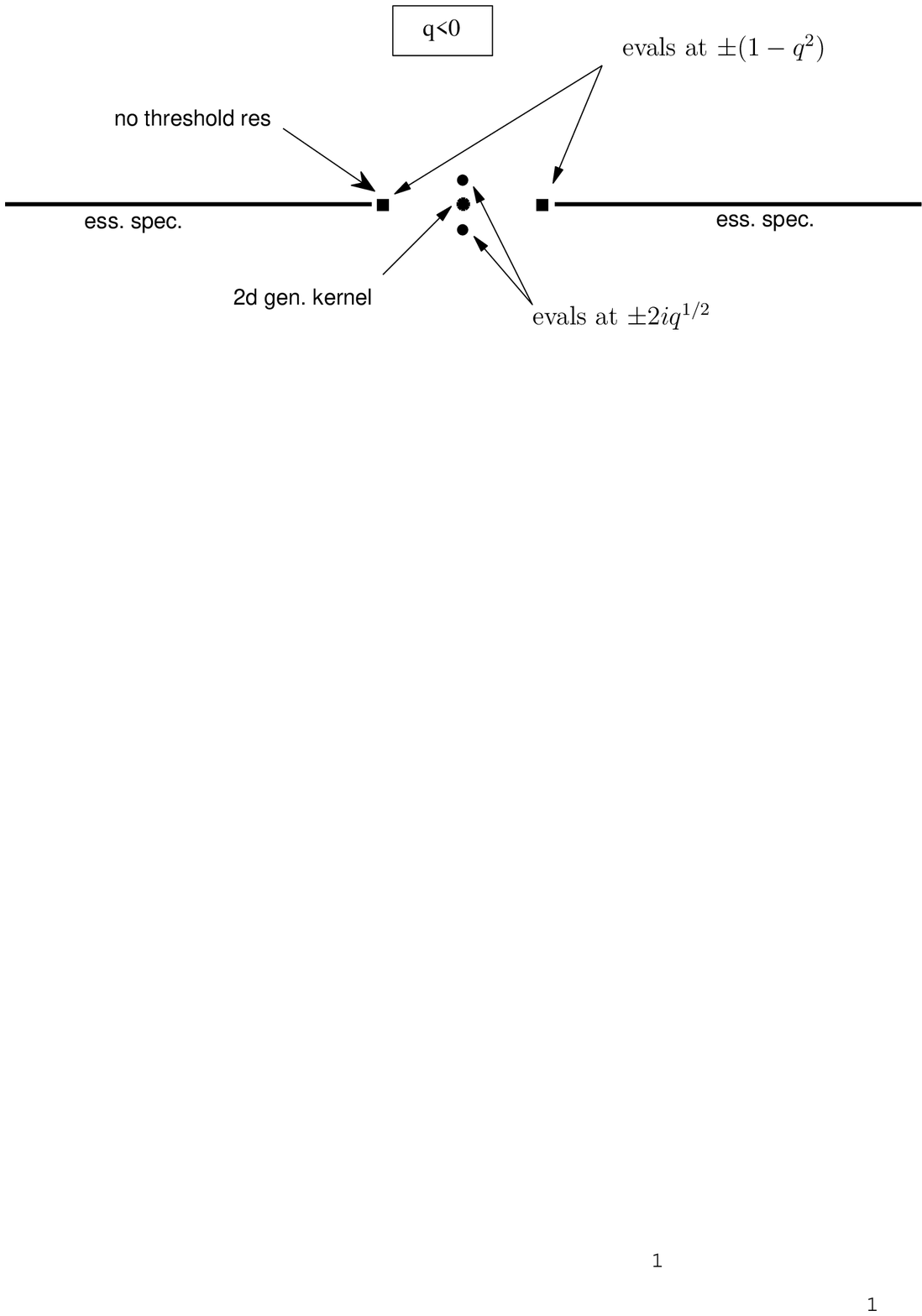}
\includegraphics[width=6in]{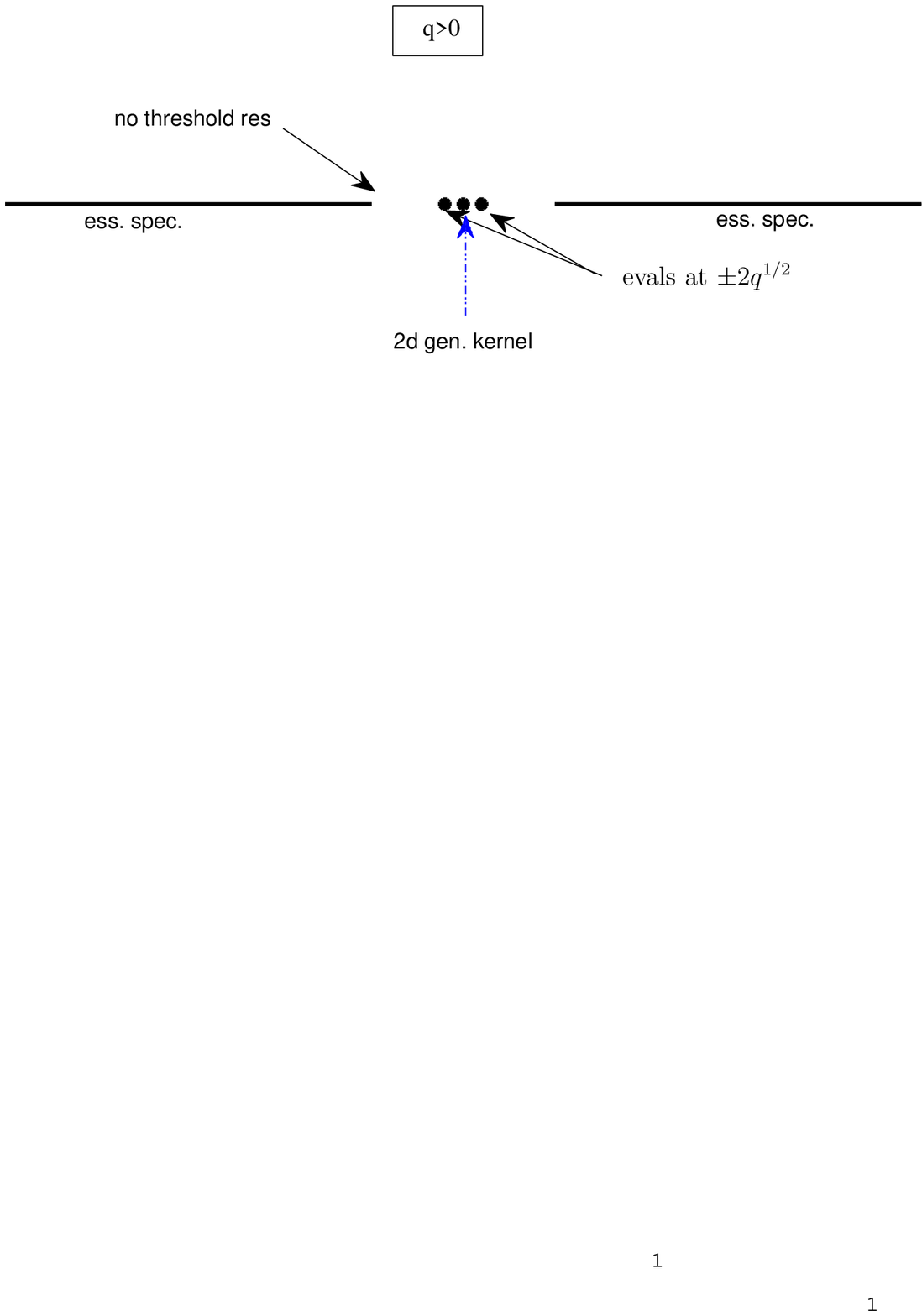}
\end{center}
\caption{The spectrum of the operator $ H_q $ for $ q < 0 $ (repulsive
$\delta$ potential) and $ q > 0 $ (attractive $\delta$ potential). 
The threshold resonances of the free problem become eigenvalues
for the repulsive potential which is counterintuitive.} 
\label{f:pcool3}
\end{figure}

We also see that there are {\em no} embedded eigenvalues in the 
continuous spectrum: they would correspond to {\em real} poles
in $ A $ and $ D $. Hence for $ k \in \RR\setminus \{ 0 \} $ 
the solutions $ \psi ( x , k ) $ and $ \psi ( x , -k ) $, or
the solutions $ \psi ( x, k ) $ and $ \psi ( - x, k ) $, form 
a basis of {\em tempered solutions} to 
\[  H_q u = ( k^2 + 1 ) u \,. \]

Our operator $ H_q $ is the Hamiltonian matrix for the quadratic form
given by 
\[  L = J H_q = 
\begin{bmatrix}
0 & -\partial_x^2 + 1 \\
-\partial_x^2 + 1 & 0 
\end{bmatrix} + 
2\sech^2(x+\sgn(x)\theta)\begin{bmatrix}
-1 & 2 \\
2  & -1
\end{bmatrix}
-2q
\begin{bmatrix}
0 & \delta_0 \\
\delta_0 & 0 
\end{bmatrix} \,, \]
but all we need are the general structural properties.

\subsection{Spectral decomposition of $ H_q $.}
\label{S:sdHq}
Let $ P^c_q $ be the symplectic projection on the symplectic
orthogonal of the discrete spectrum of $ H_q $ -- which we know 
consists of $ 4 $ eigenvalues for $ q > 0 $ and $ 6 $ eigenvalues for
$ q < 0 $.

As in the case of $ q = 0 $,
we want to write the Schwartz kernel of $ P^c_q $ as 

\[ P^c_q ( x , y )  = \frac 1 {2 \pi} \int_{\RR} 
( v_+ ( x, k  ) \tilde  v_+ ( y , k )^* 
+ v_- ( x, k  ) \tilde v_- ( y , k )^* ) dk \,, \]
where 
\[  H_q v_\pm = \pm ( k^2 + 1 ) v_\pm \,, \ \ 
 H_q^* \tilde v_\pm = \pm ( k^2 + 1 ) \tilde v_\pm \,, \]
and 
\begin{equation}
\label{eq:delta} 
\frac 1 {2 \pi}  \int_\RR \tilde v_{\pm } ( x , k )^* v_\pm ( x, k') dx = 
\delta ( k-k') \,, \ \ 
\frac 1 {2 \pi}   \int_\RR \tilde v_{\pm } ( x , k )^* v_\mp ( x, k') dx 
= 0 \,. 
\end{equation}
Now the generalized eigenfunctions are in fact double with $ \pm k $
corresponding to the single generalized eigenvalue $ k^2+1 $.

A comparison with standard one dimensional
scattering -- see \cite[(1.26),(1.30)]{TZ} -- shows that 
the states $ v_+ ( x, k ) $ should be chosen so that 
they satisfy\footnote{The states with $ \pm k > 0 $ correspond 
to $ e_\pm  $ in the notation of \cite{TZ}.} \eqref{E:v++} and \eqref{E:v+-}
-- see \cite[\S 2.2.2]{BP}
and \cite[Proposition 6.9]{KS} for a full justification of this in the case of the 
system \eqref{eq:form}.

We compare these asymptotic formul{\ae} to the
properties of $ \psi ( x, k ) $:
\[  \psi (  x,  k ) \sim \begin{bmatrix} 1 \\ 0 \end{bmatrix}
\left\{ \begin{array}{ll}   
B (  k ) ( 1 - i k )^2 e^{i k x} +  C (  k ) ( 1 + i k )^2 e^{ - i k x } \,,  & 
x \rightarrow + \infty \\
( 1 + i k )^2 e^{ixk} \,,   & x \rightarrow - \infty \end{array} \right.\,, \]
and 
\[  \psi ( - x, - k ) \sim \begin{bmatrix} 1 \\ 0 \end{bmatrix}
\left\{ \begin{array}{ll}   
( 1 - i k )^2 e^{ixk} \,,  & 
x \rightarrow + \infty \\
B ( - k ) ( 1 + i k )^2 e^{i k x} +  C ( - k ) ( 1 - i k )^2 e^{ - i k x } \,,  & 
x \rightarrow - \infty \end{array} \right.\,. \]
This shows that 
\begin{equation}
\label{eq:psi} v_+( x, k ) \defeq \left\{ \begin{array}{ll} a_+(k) 
\psi ( -x, -k ) & k > 0 \\
a_-( k ) \psi ( x, k) & k < 0 \,,  \end{array} \right. 
\end{equation}
where
\[  a_\pm ( k ) = \frac 1 { ( 1 \pm i k )^2 B ( \mp k ) } \,. \]

Note that $v_+(-x,-k)=v_+(x,k)$.  Define
\[ v_-( x, k ) \defeq \sigma_1 v_+ ( x,k ) \,, \]
and 
\[ \tilde v_\pm (x , k ) \defeq \sigma_3 
v_\pm ( x, k )  \,. \]

\subsection{Propagator $ \exp ( - i t H_q/2 ) $}

The continuous spectrum part of the propagator appearing in Theorems \ref{T:main}
and \ref{T:npt} can now be written as 
$$e^{-\frac12it H_q}P_c(x,y) = \frac1{2\pi} \int_\mathbb{R} 
( e^{-\frac12 it(1+k^2)} v_+(x,k)\tilde v_+(y,k)^* 
+ e^{
\frac12 it(1+k^2)} v_-(x,k)\tilde v_-(y,k)^* )  dk \,,  $$
where $ v_\pm $ are given in \S~\ref{S:sdHq}. We have the analogue of the first
part of 
Proposition \ref{p:w0}. Since the proof is exactly the same, it is omitted.

\begin{proposition}
\label{p:w0q}
Suppose that $ w_0 \in {\mathcal S} ( \RR \setminus \{ 0 \} ) 
\cap L^\infty ( \RR ) $ is real valued and even. Then 
\begin{gather}
\label{eq:lw0q}
\begin{gathered}
e^{-\frac 12 i t H_q } 
P_c w_0(x,t) = \frac 2 {\sqrt{2 \pi}} \int_{0}^{\infty} (a_{\textnormal{ev}}(x,k) e^{-\frac12 i(k^2+1)t} + b_{\textnormal{ev}}(x,k) e^{\frac12 i(k^2+1)t})f(k) dk \,, \\  
f(k) =\frac{2}{ \sqrt{2 \pi } }  \int_0^{\infty} 
( \overline{a_{\textnormal{ev}}(x,k)} - \overline{b_{\textnormal{ev}}(x,k)})w_0(x) \, dx \,, 
\end{gathered}
\end{gather}
where $ a $ and $ b $ are defined by 
\begin{equation}
\label{eq:ab}
v_+(x,k) 
= \begin{bmatrix} a(x, k) \\ b(x, k) \end{bmatrix}\, ,
\end{equation}
and $a_\textnormal{ev}(x,k) = (a(x,k)+a(-x,k))/2$, $b_\textnormal{ev}(x,k) = (b(x,k)+b(-x,k))/2$ .
\end{proposition}

In the above proposition, we reexpressed integrals over $\mathbb{R}$ as integrals over $(0,+\infty)$ in \eqref{eq:lw0q} using that $a(x,k)=a(-x,-k)$ and $b(x,k)=b(-x,-k)$.  This implies that $f(k)$ is even, and that $a_{\ev}(x,k)$ and $b_{\ev}(x,k)$ are even in both $x$ and $k$.

\subsection{Asymptotic analysis of the breathing patterns}
\label{aabp}

We will now prove \eqref{eq:t3} by describing the asymptotics of 
\begin{gather}
\label{eq:w00}
\begin{gathered}
 w ( x , t ) \defeq e^{-\frac 12  i t H_q } 
P_q^c w_0(x,t) \,, 
\end{gathered}
\end{gather} 
at $ x = 0 $. Here $ w_0 $ is assumed to satisfy \eqref{eq:t00}. In particular,
\begin{equation}
\label{E:w1}
w (0,t) = \sqrt \frac2{\pi} \int_0^\infty (a(0,k)e^{-\frac12 it(1+k^2)} 
+ b(0,k)e^{\frac12it(1+k^2)})f(k) \, dk \,,
\end{equation}
where $ a $ and $ b $ are defined by \eqref{eq:ab} and $ f $ is given in \eqref{eq:lw0q}.  

We now focus on the form of the expression for $f(k)$ and derive some of its smoothness and decay properties.  The behaviour of $ f $ for large values of $ k$ can be deduced directly from the definition of $f(k)$ as the pairing of $w_0$ with a solution $\psi$ to $H_q\psi = (1+k^2)\psi$.

\begin{lemma}
\label{l:fkl}
For real and even $ w_0 $ satisfying \eqref{eq:t00}, 
and for $ f $ defined by 
\eqref{eq:lw0q} we have 
$ f|_{\RR_\pm} \in C^\infty  ( \RR_\pm ) $. For $ |k| > \epsilon $  we have
\[  | f^{(p)} ( k ) | \leq \frac{ C_{\epsilon, p } q } { 1 + k^2 } \,, \]
uniformly in $ q $.
\end{lemma}
\begin{proof}
We recall from \eqref{eq:deff} (the formal structure of $ f ( k ) $ is
the same as in the free case) that 
\[ f(k) = \left(V_+^* \begin{bmatrix} 1 \\ -1 \end{bmatrix} w_0 \right)(k) 
= \frac 1 {\sqrt { 2 \pi } } \int_\RR 
v_+ ( x, k )^* \begin{bmatrix} \ w_0 ( x ) \\ -w_0( x )  \end{bmatrix} dx \,. 
\] 
The formul{\ae} for $ a ( x , k ) $ and $ b ( x , k ) $ above, and  
the formul{\ae} in Appendix \ref{AA}, show that the $ v_+( x, k ) $ 
are uniformly bounded in $ x $ and in $ k $, for $ |k| > \epsilon $.
Since $ ( k^2 + 1 ) v_+ ( x, k ) = H_q v_+ ( x, k ) $, integration by 
parts (see the formula for $ H_q $ in \eqref{eq:Hql}) shows that 
\[ \begin{split} ( 1 + k^2 ) f ( k ) & = 
\frac 1 {\sqrt { 2 \pi } } \int_\RR 
\left(
\begin{bmatrix}  - \partial_x^2 + 1 - 4 v^2 & -2 v^2 \\
-2 v^2 & - \partial_x^2 + 1 - 4 v^2 \end{bmatrix} 
v_+ ( x, k ) 
\right)^*   
\begin{bmatrix} w_0 ( x ) \\ w_0( x )  \end{bmatrix} dx  \\
& \ \ \ \ - \frac{ 2 q } { \sqrt { 2 \pi } } 
 v_+ ( 0 , k )^* \begin{bmatrix}  w_0 ( 0 ) \\ w_0( 0 )  \end{bmatrix} 
\\ 
& = \frac 1 {\sqrt { 2 \pi } } \int_{\RR \setminus 0} 
v_+ ( x, k )^* \begin{bmatrix}  - \partial_x^2 + 1 - 4 v^2 & -2 v^2 \\
-2 v^2 & - \partial_x^2 + 1 - 4 v^2 \end{bmatrix} 
\begin{bmatrix}  w_0 ( x ) \\ w_0( x )  \end{bmatrix} dx  \\
& \ \ \ \ + 
 \frac{ 1 } { \sqrt { 2 \pi } } 
 v_+ ( 0 , k )^*  \left( 
2 q 
  \begin{bmatrix}  w_0 ( 0 ) \\ w_0( 0 )  \end{bmatrix} 
+\begin{bmatrix}  w_0'(0-) - w_0' ( 0 +  )   \\ w_0'(0-) - w_0' ( 0 +  )
\end{bmatrix} \right) 
\,, 
\end{split} \]
where the last term came from the fact that 
$ w_0 ( x ) = u( x, 0 )  - v_1 ( x ) $, $ u ( x, 0 ) \in H^1  $, so that
$ w_0 $ is continuous at $ x = 0 $, and the $ w_0' ( 0 \pm ) $ terms come
from integation by parts.

The right hand side is uniformly bounded 
for $ | k | > \epsilon $ which proves the lemma for $ p = 0 $.

We can now proceed by induction noting that 
\[ \begin{split} H_q \partial^p_k v_+ ( x, k ) & 
= \partial_k^p H_q v_+ ( x , k )  \\ & =
( k^2 + 1 ) \partial_k^p v_+ ( x , k ) + 2 p k \partial_k^{p-1}  v_+ ( x, k ) +
p ( p -1) \partial_k^{p-2} v_+ ( x, k ) \,, 
\end{split} \]
and that for $ | k | > \epsilon $, 
$ | \partial^p_k v_+ ( x, k ) w_0 ( x ) | \leq C_\epsilon $,
$ x \in \RR $.
\end{proof}

We now derive a workable expression for $f(k)$.  The formul{\ae} \eqref{eq:psi0} and \eqref{eq:psi} show that for  $k>0$, we have
$$  
a(x,k) = 
\begin{aligned}[t]
&\left(\frac{(t-ik)^2e^{ik(x+\theta)} }{B(-k)(1+ik)^2} 
- \frac{A(-k)s^2e^{-\mu (x+\theta)}}{B(-k)(1+ik)^2}\right) x_+^0 \, + \ \\
&+ \left( \frac{(t-ik)^2e^{ik(x-\theta)}}{(1+ik)^2}
 + \, \frac{C(-k)(t+ik)^2e^{-ik(x-\theta)}}{B(-k)(1+ik)^2} - \frac{D(-k)s^2e^{\mu (x-\theta)}}{B(-k)(1+ik)^2} \right) x_-^0 \,, 
\end{aligned}
$$
and
$$
b(x,k) = 
\begin{aligned}[t]
& \left( \frac{-s^2e^{ik(x+\theta)}}{B(-k)(1+ik)^2} + \frac{A(-k)(t+\mu)^2e^{-\mu (x+\theta)}}{B(-k)(1+ik)^2} \right) x_+^0 \, + \ \\ 
& + \left( -\frac{s^2e^{ik(x-\theta)}}{(1+ik)^2} - \frac{C(-k)s^2e^{-ik(x-\theta)}}{B(-k)(1+ik)^2} +\frac{D(-k)(t-\mu)^2e^{\mu (x-\theta)}}{B(-k)(1+ik)^2} \right) x_-^0 \,, 
\end{aligned}
$$
where $ s = \sech ( x + \sgn ( x ) \theta ) $, $ t = \tanh( x + \sgn (x ) \theta ) $, $ \theta = 
\tanh^{-1} ( q ) $, and $ \mu = ( k^2 + 2)^{\frac 12 }$.
From these expressions, we deduce that for $x>0$, $k>0$,
$$a_{\ev}(x,k) = 
\begin{aligned}[t]
\frac{(1+C(-k))(t-ik)^2e^{ik(x+\theta)}}{2B(-k)(1+ik)^2} + \frac{(t+ik)^2e^{-ik(x+\theta)}}{2(1+ik)^2} &\\
- \frac{(A(-k)+D(-k))s^2e^{-\mu (x+\theta)}}{2B(-k)(1+ik)^2} & \,,
\end{aligned}$$
$$b_{\ev}(x,k) = -\frac{(1+C(-k))s^2e^{ik(x+\theta)}}{2B(-k)(1+ik)^2} - \frac{s^2e^{-ik(x+\theta)}}{2(1+ik)^2} + \frac{(A(-k)+D(-k))(t+\mu)^2e^{-\mu (x+\theta)}}{2B(-k)(1+ik)^2} \,,$$
and thus (using that $A(-k)=\overline{A(k)}$, etc.) for $k>0$, we have
\begin{equation}
\label{E:fksplit}
f(k) = \frac{1+C(k)}{2B(k)(1-ik)^2}f_1(k) + \frac{1}{2(1-ik)^2}f_1(-k) - \frac{A(k)+D(k)}{2B(k)(1-ik)^2} f_2(k) \,,
\end{equation}
where
\begin{equation}
\label{E:fk}
\begin{gathered}
f_1(k) = \frac{1}{\sqrt{2\pi}}\int_0^\infty ((t+ik)^2+s^2)e^{-ik(x+\theta)}w_0(x) \, dx \, ,\\
f_2(k) = \frac{1}{\sqrt{2\pi}}\int_0^\infty ((t+\mu)^2+s^2)e^{-\mu (x+\theta)}w_0(x) \, dx \,.
\end{gathered}
\end{equation}

By differentiation under the integral sign, integration by parts, and Taylor's theorem, we have

\begin{lemma}
\label{L:f-bounds}
Let $f_j$ be defined by \eqref{E:fk} and suppose $w_0$ satisfies \eqref{eq:t00}.  We have, for each $\ell = 0,1,2 \ldots$,
$$\|f_j^{(\ell)}\|_{L^\infty} \lesssim q (1+|k|)\,,$$
with the implicit constants depending only upon $\ell$ and $w_0$ (specifically, the ``$q$'' on the right side could be replaced with a finite sum of seminorms of $w_0$).  Moreover,
$$f_1(k) = f_1(0) + kf_1'(0) + k^2g(k)$$
where $g(k)$ is a smooth function satisfying 
 for each $\ell = 0,1,2 \ldots$,
$$\|g^{(\ell)}\|_{L^\infty} \lesssim q (1+|k|)\,.$$
\end{lemma}

Now we return to the computation of $w(0,t)$ given by \eqref{E:w1}.  Because of continuity at $ x = 0 $ we conclude that
\begin{equation}
\label{E:a0}
a(0,k)  = \frac{1}{(1+ik)^2} \Big( \frac{(q-ik)^2e^{ik\theta}}{B(-k)} + \frac{A(-k)(1-q^2)e^{-\mu\theta}}{B(-k)} \Big)\,, 
\end{equation}
and $b(0,k)  = b_1(k) + b_2(k)$, where
\begin{equation}
\label{E:b0}
b_1(k) = -\frac{e^{ik\theta}}{(1+ik)^2B(-k)}, \quad 
b_2(k) = \frac{q^2 e^{ik\theta} +A(-k)(q+\mu)^2 e^{-\mu\theta}}{(1+ik)^2B(-k)}
\end{equation}

Upon substituting \eqref{E:a0}, \eqref{E:b0} and \eqref{E:fksplit} into \eqref{E:w1}, we obtain an expression with many terms. We first observe in the following lemma that, fortunately, many of these terms are of lower order.

\begin{lemma}
\label{l:f1}
For $ w_0 $ satisfying \eqref{eq:t00},  and $ f_1 ( k )$, $f_2 ( k ) 
 $ defined in 
\eqref{E:fk}, we have that each of the following
$$\int_0^\infty  e^{-\frac12itk^2} a(0,k) f(k) \, dk, \quad \int_0^\infty  e^{\frac12itk^2} b_2(k) f(k) \, dk, $$
$$\int_0^\infty e^{\frac12itk^2}   \frac{b(0,k)(A(k)+D(k))}{2B(k)(1-ik)^2} f_2(k) \, dk \,,  $$
is of size 
$$\mathcal{O}\Big( \frac{q^2}{t^{1/2}} \Big) + \mathcal{O}\Big( \frac{q}{t^{3/2}}\Big) \,. $$ 
\end{lemma}

We will prove this lemma later.  In the next lemma, we deduce the asymptotic form of the dominant terms in the expression for $w(0,t)$.

\begin{lemma}
For $ w_0 $ satisfying \eqref{eq:t00},  and $ f_1 ( k )$ defined in 
\eqref{E:fk}, we have 
\label{l:f1b}
\begin{equation}
\label{E:mainexp}
\begin{aligned}
&\int_0^\infty \frac{ e^{\frac12ik^2t}}{2(1+k^2)^2B(-k)} \Big( \frac{1+C(k)}{B(k)}f_1(k) + f_1(-k) \Big) \,dk \\
&\qquad =  \frac12 t^{-1/2} e^{i\pi/4}\int_0^\infty w_0(x) \,dx  + \mathcal{O}(q^2) + \mathcal{O}\left( \frac{q}{t^{3/2}}\right)
\end{aligned}
\end{equation}
\end{lemma}

Combining Lemmas \ref{l:f1}, \ref{l:f1b}, we obtain the following proposition.

\begin{proposition}
\label{p:as}
For $ w ( 0 , t ) $ given by \eqref{E:w1} we have for $ t \gg 1 $,
\[ w ( 0 , t ) = - \sqrt {\frac{2 } {  \pi  t } } e^{ i t / 2 + i \pi / 4 } 
\int_\RR  w_0 ( x ) dx + {\mathcal O} \left( \frac{q}{t^{3/2}}
\right) + {\mathcal O} ( q^2 ) \,.
\]
\end{proposition}

\medskip
\noindent
{\bf Remark.} The leading expression in Proposition \ref{p:as} is formally 
the same as the expression in the case $ q = 0 $ in \S \ref{nlpfree}.
For the case described in Fig.~\ref{f:1}, 
$$w_0(x) = \frac{1}{1+q}\sech\left(\frac{x}{1+q}\right)-\sech(|x|+
\tanh^{-1}q)\,,$$
and we have
$$  \int_{-\infty} ^\infty w_0(x) = 2 q. $$

Now we develop some preliminaries in order to prove Lemmas \ref{l:f1} and \ref{l:f1b}.  To streamline the presentation, we introduce a definition:

\begin{definition}
\label{d:con}
A function $h(k,q)$ is \emph{conormal} (at $ k = 0 $) 
uniformly in $ q $, 
\[   h \in {\mathcal A} \,, \]
if for $ 0 \leq \ell \leq 4 $, 
\begin{equation}
\label{E:admissible}
\begin{aligned}
&| \partial_k^l h(k,q)| \leq C_\ell  \quad \ \ \ \ \ \text{for } \quad |k|\geq 1 \,,   \\
&| ( k \partial_k)^\ell h(k,q)| \leq C_\ell  \quad  \text{for} \quad |k|\leq 
1 \,, 
\end{aligned}
\end{equation}
with the constants independent of $q$.
\end{definition}

We note that the sum and product of conormal functions is conormal.  The two main types of lower order terms that we encounter arise from either $q$ times a conormal function or $k^2$ times 
a conormal function.  The former will give an error of size $q^2/t^{1/2}$ and the latter an error of size $q/t^{3/2}$.  This will follow (as we will see in more detail in the proof of Lemmas \ref{l:f1}, \ref{l:f1b} below) from Lemma \ref{L:f-bounds} and the following lemma applied with $f=f_j$, $j=1,2$ defined in \eqref{E:fk}.

\begin{lemma}
Suppose that $h(k,q)$ is conormal in the sense of Definition \ref{d:con}.  Then
\begin{equation}
\label{E:error1}
\left| \int_0^\infty e^{\pm \frac12itk^2} h(k,q) f(k) \, dk \right| \lesssim \frac{1}{\sqrt t} \sum_{j=0}^2 \|f^{(j)}\|_{L^\infty} \,,
\end{equation}
\begin{equation}
\label{E:error2}
\left| \int_0^\infty e^{\pm \frac12itk^2} k^2 h(k,q) f(k) \, dk \right| \lesssim \frac{1}{t^{3/2}} \sum_{j=0}^4 \|f^{(j)}\|_{L^\infty} \,,
\end{equation}
with the implicit constants independent of $q$.
\end{lemma}
\begin{proof}
We begin with \eqref{E:error1}.
Let $s=k\sqrt{t}$.  Then the integral to be estimated takes the form
$$\frac{1}{\sqrt t} \int_0^\infty e^{\frac12 is^2} h\Big(\frac{s}{\sqrt t}\Big) f\Big(\frac{s}{\sqrt t} \Big) \, ds$$
Let $ \chi( s )  $ satisfy
\begin{equation}
\label{eq:chi}
\chi \in C_{\rm{c}}^\infty ( (-1 , 1 ) )\,, \ \ 
 \text{ $\chi $  is equal to $ 1 $ in a neighbourhood of $ s = 0 $.}
\end{equation}
Clearly,
$$\left| \frac{1}{\sqrt t} \int_0^\infty \chi(s) e^{\frac12 its^2} h\Big(\frac{s}{\sqrt t}\Big) f\Big(\frac{s}{\sqrt t} \Big) \, ds \right| \lesssim \frac{\|h\|_{L^\infty}\|f\|_{L^\infty}}{\sqrt t} \lesssim \frac{\|f\|_{L^\infty}}{\sqrt t}$$
and therefore we just need to estimate
\begin{equation}
\label{E:s-large}
\frac{1}{\sqrt t} \int_0^\infty (1-\chi(s)) e^{\frac12 is^2} h\Big(\frac{s}{\sqrt t}\Big) f\Big(\frac{s}{\sqrt t} \Big) \, ds
\end{equation}
Using that $(-is^{-1}\partial_s)^2 e^{\frac12is^2} = e^{\frac12is^2}$ and two applications of integration by parts gives 
$$\frac{1}{\sqrt t} \int_0^\infty e^{\frac12 i s^2} (-i\partial_s \, s^{-1})^2 \Big[ (1-\chi(s))h\Big( \frac{s}{\sqrt t}\Big) f\Big(\frac{s}{\sqrt t}\Big) \Big] \, ds$$
Distributing the derivatives and estimating (using the $s^{-2}$ factor to carry out the integration), we obtain the bound
$$ \left( \sum_{0\leq \ell \leq 2} \frac{\|h^{(\ell)}\|_{L^\infty(|k|\geq t^{-1/2})}}{t^{\ell/2}} \right)  \left( \sum_{0\leq \ell \leq 2} \frac{\|f^{(\ell)}\|_{L^\infty(|k|\geq t^{-1/2})}}{t^{\ell/2}} \right) \,.$$  
Now we just apply \eqref{E:admissible} to obtain the bound \eqref{E:error1}.

Now we establish \eqref{E:error2}.  Let $s=k\sqrt t$ to obtain
$$\frac{1}{t^{3/2}} \int_0^\infty e^{\frac12 is^2} s^2 h \Big( \frac{s}{\sqrt t} \Big) f_j \Big( \frac{s}{\sqrt t}\Big) \, ds$$
The remainder of the proof is similar to that above, except that we need to use $(-is^{-1}\partial_s)^4 e^{\frac12is^2} = e^{\frac12is^2}$ and four applications of integration by parts.
\end{proof}

We shall need the following properties of the scattering coefficients $A$, $B$, $C$, and $D$, obtained from the more precise asymptotics in Appendix \ref{AA}.

\begin{lemma}[Properties of $A$, $B$, $C$, $D$]
\label{L:asymp} 
$1/B(k)$ and $C(k)/B(k)$ are conormal, and in fact
\[ \begin{split}
& \frac{1}{B(k)} = \frac{k}{k-iq} + q \alpha_1 ( k , q )   + k^2 
\alpha_2 ( k , q )  \,,\\ 
& \frac{C(k)}{B(k)} = \frac{iq}{k-iq} +  
 q \alpha_3 ( k , q )   + k^2 
\alpha_4 ( k , q )   \,, \end{split} \]
where $  \alpha_j \in {\mathcal A} $ are conormal 
in the sense of Definition \ref{d:con}.
Also, 
\[ \frac{A( k ) }{ B ( k ) } = 
  q \beta_1 ( k , q )  \,, \ \ 
\frac{D ( k ) }{ B ( k ) } = 
q \beta_2 ( k , q ) \,, \ \ \beta_j \in {\mathcal A} \,. \]
\end{lemma}

With these preliminaries out of the way, 
we can now prove Lemmas \ref{l:f1} and \ref{l:f1b}.

\begin{proof}[Proof of Lemma \ref{l:f1}]
We shall give the proof for
\begin{equation}
\label{E:a0f}
\int_0^\infty e^{-\frac12itk^2} a(0,k) \, f(k) \, dk \,.
\end{equation}
The other integrals in the statement of the lemma are treated similarly.  By 
Lemma \ref{L:asymp}, we see that in the expression \eqref{E:fksplit}, all coefficients of $f_1$, $f_2$ are conormal.  Also by Lemma \ref{L:asymp} and \eqref{E:a0}, we see that 
$$a(0,k) =  q a_1 ( k, q  ) + k^2 a_2 ( k , q )  \,, \ \  a_j \in {\mathcal A} 
\,.$$
By the alegbra property of the conormal class, \eqref{E:error1},\eqref{E:error2}, and Lemma \ref{L:f-bounds}, we obtain that \eqref{E:a0f} is of size $\mathcal{O}(q^2/t^{1/2})+\mathcal{O}(q/t^{3/2})$. 
\end{proof}

\begin{proof}[Proof of Lemma \ref{l:f1b}] 

We write $\approx$ to mean that the two quantities are equal with an error of the form $q$ times conormal or $k^2$ times conormal.  By Lemma \ref{L:asymp},
$$\frac{1}{B(-k)} \approx \frac{k}{k+iq} \,, \quad \frac{1+C(k)}{B(k)} \approx \frac{k+iq}{k-iq} \,, \quad \frac{1}{(1+k^2)^2} \approx 1 \,.$$ 
We also take the expansion in Lemma \ref{L:f-bounds}:
\begin{gather*}
f_1(k) = f_1(0) + kf'(0) + k^2g(k) \,,\\
f_1(-k) = f_1(0) - kf'(0) + k^2g(-k) \,.
\end{gather*}
Substituting the above into \eqref{E:mainexp} and appealing to \eqref{E:error1},\eqref{E:error2} and Lemma \ref{L:f-bounds} for the error terms, we see that \eqref{E:mainexp} is equal to 
$$
\begin{aligned}
\int_0^\infty e^{\frac12itk^2}\frac{k}{2(k-iq)} \left( \frac{k+iq}{k-iq}(f_1(0)+kf'(0))+ (f_1(0)-kf_1'(0)) \right) \, dk  &\\
 + \mathcal{O}\left( \frac{q^2}{t^{1/2}}\right)+ \mathcal{O}\left( \frac{q}{t^{3/2}}\right)  &\,.
\end{aligned}$$
This simplifies to
\begin{equation}
\label{E:mainexp2}
\begin{aligned}
f_1(0) \int_0^\infty e^{\frac12 it  k^2} \frac{k^2}{k^2+q^2} \, dk + iqf_1'(0) \int_0^\infty e^{\frac12 itk^2}\frac{k^2}{k^2+q^2}\, dk & \\
+ \mathcal{O}\left( \frac{q^2}{t^{1/2}}\right) + \mathcal{O}\left( \frac{q}{t^{3/2}}\right) & \,.
\end{aligned}
\end{equation}
Note that
$$\int_0^\infty e^{\frac12 i s^2} \frac{s^2}{s^2+\delta^2} \, ds = \int_0^\infty e^{\frac12 i s^2} \, ds - \delta \int_0^\infty e^{\frac12 i \delta^2 s^2} \frac{1}{s^2+1} \, ds \,,$$
where, in the second term, we made the substitution $s\mapsto \delta s$.  Thus,
$$\int_0^\infty e^{\frac12 i s^2} \frac{s^2}{s^2+\delta^2} \, ds = \sqrt \frac{\pi}{2}e^{i\pi/4} + \mathcal{O}(\delta) \,.$$
In \eqref{E:mainexp2}, make the substitution $s=t^{1/2}k$ and appeal to the above formula to obtain
\begin{align*}
&\frac{f_1(0) + \mathcal{O}(q^2)}{\sqrt t} \left( \sqrt \frac{\pi}{2}e^{i\pi/4} + \mathcal{O}(q\sqrt t)\right) + \mathcal{O}\left( \frac{q^2}{t^{1/2}}\right) + \mathcal{O}\left( \frac{q}{t^{3/2}}\right) \\
&\quad = \frac12 t^{-1/2} e^{i\pi/4}\int_0^\infty w_0(x) \,dx  + \mathcal{O}(q^2) + \mathcal{O}\left( \frac{q}{t^{3/2}}\right) \,.
\end{align*}
\end{proof}

\subsection{Proof of Theorem \ref{T:main}}
\label{ptm}
We will now combine Theorem \ref{T:npt} with the results of this section to 
proof Theorem \ref{T:main}.  We start with the following lemma 

\begin{lemma}
\label{l:w0l}
Suppose that $ w_0 $ satisfies the assumptions of Theorem \ref{T:main} 
and that 
\[ \lambda_0 \defeq 1 + \int_\RR w_0 ( x ) v_1 ( x ) dx \]
is the nonlinear eigenvalue specified in Theorem \ref{T:main}. 
Then 
for the projection $ P_{\lambda_0} $ defined by \eqref{eq:proj} 
(with $ q $ supressed in the subscript),
\begin{equation}
\label{eq:lpl}
  P_{\lambda_0 } ( v_{\lambda_0} - v_1 - w_0 ) = {\mathcal O} ( q^2 ) \,,
\end{equation}
and consequently the solution, $ \lambda $, to 
$ P_\lambda ( v_\lambda - v_1 - w_0 ) = 0 $ satisfies
\begin{equation}
\label{eq:lpl1} 
 \lambda - \lambda_0  = {\mathcal O} ( q^2 ) \,.
\end{equation}
\end{lemma}
\begin{proof}
The definition \eqref{eq:proj} means that we need to show
that $ \omega ( v_{\lambda_0 } - v_1 - w_0, i v_{\lambda_0}  ) = {\mathcal O} ( q^2) $ since the other term 
vanishes by the reality of $ w_0 $. Now, using the definition of $ \lambda_0 $
and the fact that  $\lambda_0 = 1 + {\mathcal O} ( q ) $, we see that
\[ \begin{split}
 \omega ( v_{\lambda_0 } - v_1 - w_0, i v_{\lambda_0 } ) & = 
2 ( \lambda_0 - q ) - \int v_{\lambda_0 } ( x ) v_1 ( x ) dx - 
\int w_0 ( x ) v_{\lambda_0 } ( x) dx \\
& = 2 ( \lambda_0 - q ) - \int v_{\lambda_0 } ( x ) v_1 ( x ) dx  - 
\int w_0 ( x ) v_1 ( x ) dx + {\mathcal O} ( q^2) \\
& = \int w_0 ( x ) v_1 ( x ) - ( \lambda_0 - 1 ) \int \partial_\lambda(
v_\lambda )|_{\lambda = 1 } ( x ) v_1 ( x ) dx + {\mathcal O} ( q^2) \,. 
\end{split} \]
The estimate \eqref{eq:lpl} follows from 
\[ \int \partial_\lambda( v_\lambda )|_{\lambda = 1 } ( x ) v_1 ( x ) dx  = 
\frac 1 2 \partial_{\lambda} \| v_\lambda \|_{L^2}^2 |_{\lambda =1 } = 1 \,.
\] 
The comparison \eqref{eq:lpl1} 
between the exact solution and the approximate one
is obtained from the implicit function theorem as in the proof 
of Proposition \ref{p:eigen-select}. 
\end{proof}

The lemma shows that 
the assumptions of Theorem \ref{T:npt} are satisfied for $ h = C q $, 
$ \theta = 0 $ ($w_0 $ is real) and 
\[  \lambda = \lambda_0 + {\mathcal O} ( q^2 ) \,.\]
We can then apply Corollary \ref{C:npt} to obtain
\begin{equation}
\label{eq:recall} \|  u (t) - e^{ it\lambda_0^2 /2 } \Big( v_{\lambda_0 }
+ e^{-it\mathcal{L}_{\lambda_0,q}}w_0 \Big) \|_{H_x^1} 
\leq C (1+t)^2 q^2 \,,  
\end{equation}
for all $0\leq t\ll  q^{-1/2}$. We now write 
\[ \begin{split}
  e^{-it\mathcal{L}_{\lambda_0,q}}w_0 & = \begin{bmatrix} 1 & 0 \end{bmatrix}
e^{-\frac 12 i t H_{\lambda_0, q } } 
\begin{bmatrix} w_0 \\ w_0 \end{bmatrix} \\
& = \begin{bmatrix} 1 & 0 \end{bmatrix} e^{-\frac 12 i t H_{\lambda_0, q } } 
P_d \begin{bmatrix} w_0 \\ w_0 \end{bmatrix} +
 \begin{bmatrix} 1 & 0 \end{bmatrix} e^{-\frac 12 i t H_{\lambda_0, q } } 
P_c \begin{bmatrix} w_0 \\ w_0 \end{bmatrix} \,. 
\end{split} \]
The first conclusion of Theorem \ref{T:main} given in \eqref{eq:t0} 
is immediate from \eqref{eq:recall} and Proposition \ref{p:w0q} once we show 
that
\begin{equation}
\label{eq:Pd}
 \begin{bmatrix} 1 & 0 \end{bmatrix} e^{-\frac 12 i t H_{\lambda_0, q } } 
P_d \begin{bmatrix} w_0 \\ w_0 \end{bmatrix}  = {\mathcal O}_{H^1}
( q^{3/2} ) \,, \ \ 0 \leq t \ll q^{-1/2}  \,. 
\end{equation}
For $ q > 0 $ we have six contributions to the discrete spectrum, while for 
$ q < 0 $ there are four. By Lemma \ref{l:B} the non-zero eigenvalues 
in the neighbourhood of zero do not contribute as they are odd while $ w_0 $
is even. The contribution of the zero eigenvalues is $ {\mathcal O} ( q^2 t ) $
by the same arguments as in \S \ref{npt}. For $ q < 0 $ the coefficients of
the eigenfuctions (which are uniformly bounded in $ H^1 $) are estimated 
using Lemma \ref{l:reig} by 
\[   C q^{1/2} \int_\RR |w_0 ( x ) | e^{-q|x|} \, dx \leq C' q^{3/2} \,. \]
Hence \eqref{eq:Pd} holds and in view of \eqref{eq:recall} we have established 
\eqref{eq:t0}.

To obtain \eqref{eq:t3} we use the above estimate and Proposition \ref{p:as}.
The combined error term for $ 0 \ll t \ll q^{-1/2} $ is 
\[   {\mathcal O} \left( \frac {|q| } {t^{3/2} }  + |q|^{3/2} + q^2 t^2 
\right) \,, \]
and that is bounded by $ C |q| / t^{3/2} $  for $ t \leq C' |q|^{-2/7} $.


\appendix

\section{A derivation of Kaup's basis using MATLAB}
\label{ndkp}

The Kaup spectral decomposition of the linearized operator
was based on the connection with the Zakharov-Shabat system
and the complete integrability of the cubic NLS, see \cite{Kaup76} 
and \cite{Ya}. We rediscovered the structure of his 
basis of solutions through a numerical experiment and it might
be of interest to indicate how that was done. The original 
motivation was to show that 
the threshold resonances for the linearization of 
the cubic nonlinear Schr\"odinger equation (NLS) on the line are 
{\em simple} which can be done by an explicit construction of a 
solution to a system of ODEs. 

The explicit resonant state of the linearized operator $ H_0 
$ at $ 1 $ is given by 
\begin{equation}
\label{eq:R1} u_1 = \begin{bmatrix} 1 - \sech^2 x \\
 - \sech^2 x  \end{bmatrix} \,.
\end{equation}
To show that it is simple, 
we need to show that any other bounded solution is a multiple of $ u_1 $.

As in standard scattering theory, the four independent solutions of 
$ H u = u /2 $ can be characterized by their behaviour as $ x \rightarrow 
\infty $ -- see the proof of \cite[Lemma 5.19]{KS}. In particular
the resonant states can only be given as linear combinations of the two 
solutions, $ u_1 $ and $ u_2 $, satisfying
\begin{equation}
\label{eq:u}
  u_1 = \begin{bmatrix} 1 \\ 0 \end{bmatrix}  + {\mathcal O} ( e^{-2 x } ) 
\,, \ \ 
e^{ \sqrt 2 x } u_2 = \begin{bmatrix} 0 \\ 1  \end{bmatrix} 
 + {\mathcal O} ( e^{-2 x } ) \,, \ \ 
x \longrightarrow +\infty \, . 
\end{equation}
We see that $ u_1 $ is given by \eqref{eq:R1} and 
$ u_1 \in L^\infty $. 
 Once we show that $ u_2 \notin L^\infty $ 
we will see that the multiplicity of the resonance is one.

As we have already seen in \S \ref{sa} the solution $ u_2 $ can 
be written explicitly. An elementary calculation confirms that 
\begin{equation}
\label{eq:U2}  u_2 = \frac{ \exp ( - \sqrt 2 x ) } 
{ ( 1 +  \sqrt 2 )^2} \begin{bmatrix}  - \sech^2 x \\
(  \tanh x + \sqrt 2 )^2  \end{bmatrix} \,. 
\end{equation}
This shows that 
\[ u_2 = \exp ( - \sqrt 2 x ) \frac
{ (-1 + \sqrt 2)^2 } { (1 +  \sqrt 2)^2  } \left( 
\begin{bmatrix} 0  \\
1 \end{bmatrix} + {\mathcal O} ( e^{-2 |x| } ) \right) \,, \ \ 
x \longrightarrow - \infty \,,  \]
and, in particular, that $ u_2 \notin L^\infty $.

The 
exact expression \eqref{eq:U2} was arrived at through an attempt to 
produce a computer assisted proof of the fact that $ u_2 \notin L^\infty $.

The solution $ u_2  ( x ) $ is obtained by solving the following
Volterra integral equation (see \cite[(5.4)]{KS}, where one 
should let $ \lambda \rightarrow 0 $ and renormalize following 
$ \partial_x^2 \mapsto \partial_x^2/2 $):
\[ u_2 ( x ) = e^{-\sqrt 2 x } \begin{bmatrix} 0 \\ 1 \end{bmatrix}
-  {\sqrt{2}} \int_x^\infty \begin{bmatrix} 2\sqrt 2 ( y - x ) & \sqrt 2 (y - x) \\
\sinh( \sqrt 2 ( y - x ) ) & 2 \sinh ( \sqrt 2 ( y - x ) ) 
\end{bmatrix} \sech^2 y \, u_2 ( y ) dy \,. \]
Solving this equation by iteration could in principle show that the 
solution is unbounded. 

Elimination of the exponential growth in the integral equation is
helpful for the theoretical estimates of \cite{BP} and \cite{KS} and 
seems essential for a succesful numerical scheme. Thus we 
consider
\[ v ( x ) \defeq \exp ( \sqrt 2 x ) u_2 ( x ) \,, \]
which solves
\begin{gather}
\label{eq:V}
\begin{gathered}
 v ( x ) =  \begin{bmatrix} 0 \\ 1 \end{bmatrix} +  {\mathcal K} v  ( x ) \,, \ \ 
{\mathcal K} v ( x ) \defeq
\int_x^\infty 
K ( x , y ) v ( y ) dy \,, \\ 
K ( x , y ) \defeq
- \sqrt 2  \begin{bmatrix} 2\sqrt 2 ( y - x ) & \sqrt 2 (y - x) \\
\sinh( \sqrt 2 ( y - x ) ) & 2 \sinh ( \sqrt 2 ( y - x ) ) 
\end{bmatrix} \sech^2 y \, \exp ( \sqrt 2 ( x - y ) ) \,. 
\end{gathered}
\end{gather}
It is not hard to see the convergence of 
\begin{equation}
\label{eq:Kn}  v ( x) = \sum_{n=0}^\infty  {\mathcal K}^n 
\left( \begin{bmatrix} 0 \\ 1 \end{bmatrix} \right) ( x) \,, 
\end{equation}
in, say $ C^k ( \RR ) $, for any $ k $. Hence showing that 
$ v ( x ) \not \rightarrow 0 $, $ x \rightarrow -\infty $ is 
in principle possible by a numerical computation. 

We easily implement the operator $ {\mathcal K } $ in 
MATLAB. The input is an array 
which is a discretised $\RR^2$-valued function on $ [-10, 30 ] $ with $ N $ 
grid points. Because of the Volterra structure of the equation the 
left limit, $ -10 $, is not important. The right cutoff, $ 30 $, 
is chosen large enough to make the effect of the potential 
negligible. The integrals are computed using the built-in
trapezium rule and the errors can be estimates. We used $ N = 10^4 $
which would have to be even larger for rigorous estimates, while
experimentally it was clearly an ``overkill''.

{\tt function KT = KT(u)

mu =sqrt(2);

[M,N]=size(u);

x = linspace(-10,30,N);

for j=1:N-1

y = linspace(x(j),30,N-j+1); v = sech(y).*sech(y);

u1=u(:,[j:N]);

uu(1,:)=v.*(y-x(j)).*(2*u1(1,:)+u1(2,:));

uu(2,:)=v.*(sinh(mu*(y-x(j)))/mu).*(u1(1,:)+2*u1(2,:));

uu(1,:)=-4*exp(mu*(x(j)-y)).*uu(1,:);

uu(2,:)=-4*exp(mu*(x(j)-y)).*uu(2,:);

KT(1,j)=trapz(y,uu(1,:));

KT(2,j)=trapz(y,uu(2,:));

clear uu

end

KT(:,N)=[0;0];}

When the numerical solution obtained using \eqref{eq:Kn} with $ n = 10 $ 
was plotted (see Fig.\ref{f:3}) we noticed that the plot of the first
component looked remarkably like a plot of $ - \alpha \sech^2 x $, 
$ \alpha > 0 $ and the fit based on the minimum of first component
(experimental $ - \alpha $) was almost exact. From the operator $ H_0 $ it
is clear that having one component of the solution we obtain the other
and that quickly led to the exact solution \eqref{eq:U2}.  This 
then suggests  the form of general solution for other values of $ k $ and 
$ \mu $ as given in \S \ref{sa}.

It would be difficult in general, and by our
method in particular, to show the existence of a resonance. 


\begin{figure}
\begin{center}
\includegraphics[width=6in]{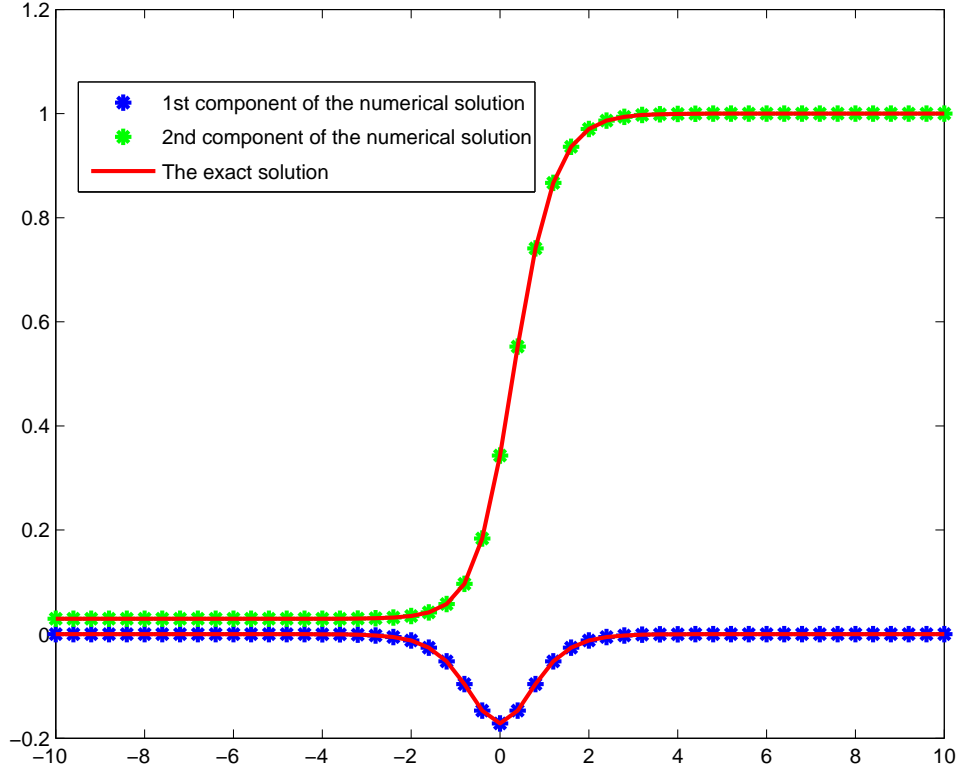}
\end{center}
\caption{The plots of components of $ \exp ( \sqrt 2 x ) u_2 ( x ) $,
following the numerical computation and the exact solutions. We
see that $ \lim_{x\rightarrow -\infty } \exp ( \sqrt 2 x ) u^2_2 ( x ) = 
( 3 - 2 \sqrt 2 ) / ( 3 + 2 \sqrt 2 ) \neq 0 $, where $ u_2^2 $ is the 
second component of the solution. We used $ N = 10^4 $ grid point and 
the plot shows the sampling of $ 100 $ points.}
\label{f:3}
\end{figure}

\section{Perturbation of eigenvalues at zero energy}
\label{S:pnz}

Here we present the perturbation theory for $ H_q - z $ at $ z = 0 $. 
Even though we could in principle obtain the same results from careful analysis
of the matrix $ \mathcal{A} ( k, q ) $ described in detail in Appendix \ref{AA},
the method used here is more general and does not depend on explicit
formul{\ae}. It is of course close to the similar study in the semiclassical 
case, see \cite{GaSi} and references given there. However, since the delta
function is clearly different from a slowly varying potential with a nondegenerate minimum we give a selfcontained argument.

\subsection{Grushin problem}

We recall that the linearized operator acting on 
$$ \begin{bmatrix} \Re w \\ \Im w \end{bmatrix} \in \RR^2 \subset \CC^2  \,, 
$$ is
given by 
\begin{equation}
\label{E:Rq}
F_q \defeq \begin{bmatrix} 0 & 1 \\ -1 & 0 \end{bmatrix}
\begin{bmatrix} L_{q+} & 0 \\ 0 & L_{q-} \end{bmatrix} \, , \qquad
\begin{aligned}
& L_{q+} = \tfrac12(1-\partial_x^2 - 6v^2 -2q\delta_0) \\
& L_{q-} = \tfrac12(1-\partial_x^2 - 2v^2 -2q\delta_0)  \,, 
\end{aligned}
\end{equation}
where $ v $ is the nonlinear ground state.
We take elements of $H^2(\mathbb{R};\mathbb{C})$ and write them as column 
vectors of real and imaginary parts giving an identification
$$H^2(\mathbb{R};\mathbb{C}) \simeq
H^2(\mathbb{R};\mathbb{R})\oplus H^2(\mathbb{R};\mathbb{R})$$
The elements $e_j\eta$ take the 2-vector form
$$e_1\eta = \begin{bmatrix} -\eta \\ 0 \end{bmatrix} , \quad e_2\eta 
= \begin{bmatrix} 0 \\ x\eta \end{bmatrix}, \quad e_3\eta 
= \begin{bmatrix} 0 \\ \eta \end{bmatrix}, \quad e_4\eta 
= \begin{bmatrix} \eta + x\eta' \\ 0 \end{bmatrix}$$
The symplectic form, in vector notation, becomes
\begin{equation}
\label{E:omega}
\omega\left( \begin{bmatrix} u_1 \\ u_2 \end{bmatrix}, \; 
\begin{bmatrix} v_1 \\ v_2 \end{bmatrix} \right) = \int (-u_1v_2+v_1u_2) 
\end{equation}
In the matrix notation, the relations \eqref{E:kernel1}
become 
$$\begin{bmatrix} 0 & L_- \\ -L_+ & 0 \end{bmatrix} \begin{bmatrix} -\eta' \\ 
0 \end{bmatrix} = 0, \qquad 
\begin{bmatrix} 0 & L_- \\ -L_+ & 0 \end{bmatrix} \begin{bmatrix} 0 \\ 
x\eta \end{bmatrix} = \begin{bmatrix} -\eta' \\ 0 \end{bmatrix}
$$
$$\begin{bmatrix} 0 & L_- \\ -L_+ & 0 \end{bmatrix} \begin{bmatrix} 0 \\ 
\eta \end{bmatrix} = 0, \qquad 
\begin{bmatrix} 0 & L_- \\ -L_+ & 0 \end{bmatrix} 
\begin{bmatrix} x\eta'+\eta \\0 \end{bmatrix} = \begin{bmatrix} 0 \\ 
\eta \end{bmatrix}
$$

To perform spectral analysis, we complexify the space and work on
$$H \defeq H^2(\mathbb{R};\mathbb{C})\oplus H^2(\mathbb{R};\mathbb{C})$$
The symplectic form $\omega$ \eqref{E:omega} extends to $H$, by analytic 
continuation (with exactly the same expression as in \eqref{E:omega}; we do 
not insert any complex conjugations).

Following the standard procedure (see \cite{SZ}) we build an {\em invertible}
matrix in block form
$$\mathcal{G}_q = \begin{bmatrix} F_q-z & R_- \\ R_+ & 0 \end{bmatrix}$$
with suitably chosen
\begin{align*}
R_- : \mathbb{C}^2 \to H \,, \ \ R_+ : H \to \mathbb{C}^2\,. 
\end{align*}
We will select $R_-$, $R_+$ to be constant (independent of $q$ and $z$) 
operators such that $\mathcal{G}$ is invertible with inverse represented in 
block form as
$$\mathcal{G}^{-1} = \begin{bmatrix} E & E_+ \\ E_- & E_{-+} \end{bmatrix}$$
The components depend on $q$ and $z$ and have the following mapping properties:
\begin{align*}
& E: H \to H \,, 
&& E_-: H \to \mathbb{C}^2 \,, \\
& E_+: \mathbb{C}^2 \to H \,, 
&& E_{-+}: \mathbb{C}^2 \to \mathbb{C}^2\,.
\end{align*}
To find $ R_-$ and $R_+$ and to compute $E^0(z)$, 
$E_+^0(z)$, $E_-^0(z)$ and $E_{-+}^0(z)$, we first consider 
$ (F_0  - z ) |_{ {\mathfrak g} \cdot \eta } $, that is $ F_0 - z $ 
acting on the generalized kernel.
Ordering the basis using $ e_j\cdot \eta $, 
$ j = 1, 2, 3, 4 $, we see that 
\[ (F_0  - z ) |_{ {\mathfrak g} \cdot \eta }  = 
\begin{bmatrix}
-z  &  1  &    &     \\
 0  & -z  &    &     \\
    &     & -z &  1  \\
    &     &  0 & -z  
\end{bmatrix}\,. \]
The computation (see \cite[\S 2.2]{SZ})
$$
\begin{bmatrix}
-z  &  1  &    &    & 0 & 0 \\
 0  & -z  &    &    & 1 & 0 \\
    &     & -z &  1 & 0 & 0 \\
    &     &  0 & -z & 0 & 1 \\
1   &  0  & 0  &  0 &   &   \\
0   &  0  & 1  &  0 &   & 
\end{bmatrix}^{-1} =
\begin{bmatrix}
0   & 0   &    &    & 1  & 0 \\
1   & 0   &    &    & z  & 0 \\
    &     &  0 & 0  & 0  & 1 \\
    &     &  1 &  0 & 0  & z \\
z   &  1  & 0  &  0 & z^2&   \\
0   &  0  & z  &  1 &    & z^2
\end{bmatrix}
$$
gives us $ R_\pm $ for which $ {\mathcal G}_0 $ is invertible:
\[ R_- \begin{bmatrix}  \zeta_1 \\ \zeta_2 \end{bmatrix} 
  = \zeta_1 e_2 \eta + \zeta_2 e_4 \eta \,, \ \
R_+ u = \begin{bmatrix} P_1  u  \\ P_3 u \end{bmatrix} \,, \ \  P u = \sum_{j=1}^4 P_j u \, e_j \eta \,. \]

This tells us that 
$$E_+^0 \begin{bmatrix} \zeta_1 \\ \zeta_2 \end{bmatrix} 
= \begin{bmatrix} \zeta_1 \\ z\zeta_1 \\ \zeta_2 \\ z\zeta_2 \end{bmatrix} 
= \zeta_1 e_1 \eta + z\zeta_1 e_2 \eta + \zeta_2e_3\eta + z\zeta_2 e_4\eta$$
or, more explicitly,
\begin{equation}
\label{E:E+}
E_+^0 \begin{bmatrix} \zeta_1 \\ \zeta_2 \end{bmatrix} 
= \begin{bmatrix} -\eta' & z(\eta+x\eta') \\ zx\eta & \eta \end{bmatrix} 
\begin{bmatrix} \zeta_1 \\ \zeta_2 \end{bmatrix}
\end{equation}
We also find that 
$$E_-^0 \begin{bmatrix} \alpha_1 \\ \alpha_2 \\ 
\alpha_3 \\ \alpha_4 \end{bmatrix} 
= \begin{bmatrix} z\alpha_1 + \alpha_2 \\ z\alpha_3 + \alpha_4 \end{bmatrix}$$
or in other words $E_-^0:H \to \mathbb{C}^2$ is expressed as
\begin{equation}
\label{E:E-}
E_-^0 \begin{bmatrix} u \\ v \end{bmatrix} 
= \begin{bmatrix} -z\int ux\eta -\int v\eta' \\ 
z\int v(x\eta)' - \int u\eta \end{bmatrix}
\end{equation}
We use the following formula to compute $E_{-+}^q(z)$:
$$E_{-+}^q = E_{-+}^0 - E_-^0(F_q-F_0)E_+^0 + \mathcal{O}(q^2)$$
By the Schur complement formula, $F_q-z$ is invertible if and only if 
$E_{-+}^q(z)$ is 
invertible, so we want to find $z$ (in terms of $q$) such that 
$\det(E_{-+}^q(z)) =0$.  We know that $E_{-+}^0:\mathbb{C}^2\to \mathbb{C}^2$ 
is
$$E_{-+}^0 = \begin{bmatrix} z^2 & 0 \\ 0 & z^2 \end{bmatrix}$$
and thus have all the ingredients to analyze the perturbation.

\subsection{Substitutions}
We have
$$\sech^2(|x|+q) = \sech^2x -2q\sech^2x\tanh|x| + \cdots$$
and thus (to first order in $q$)
\begin{align*}
L_+^q - L_+^0 = 6q\sech^2x\tanh|x| -q\delta_0(x)\\
L_-^q - L_-^0 = 2q\sech^2x\tanh|x| -q\delta_0(x)
\end{align*}
and therefore (to first order in $q$),
$$F_q-F_0 = \begin{bmatrix} 0 & 2q\sech^2x\tanh|x| -q\delta_0(x) \\ 
-6q\sech^2x\tanh|x| +q\delta_0(x) & 0 \end{bmatrix} \,. $$

We will use the notation $\eta=\sech x$ and $\sigma =\tanh x$.  
Using \eqref{E:E+} we see that $(F_q-F_0)E_+^0: \mathbb{C}^2 \to H$ takes the form
$$(F_q-F_0)E_+^0 = \begin{bmatrix}
2qz x\eta^3\sigma\sgn x  & 2q\eta^3\sigma \sgn x- q\delta_0 \\
-6q\eta^3\sigma^2\sgn x & -6qz\eta^3\sigma(1-x\sigma)\sgn x + qz\delta_0 
\end{bmatrix}
$$
From this, and \eqref{E:E-}, 
we compute $E_-^0(F_q-F_0)E_+^0:\mathbb{C}^2 \to \mathbb{C}^2$ 
takes the form
$$E_-^0(F_q-F_0)E_+^0 = \begin{bmatrix}
-qz^2\alpha -q\beta  & 0 \\
0  & q\gamma + qz^2\delta
\end{bmatrix}
$$
where 
\begin{align*}
\beta = 6\int \eta^4\sigma^3\sgn x = 1  \,, \ \ 
\gamma = 1-2 \int \eta^4\sigma \sgn x =0 \,, 
\end{align*}
and thus
\begin{equation}
\label{eq:Eq}
E_{-+}^q(z) = \begin{bmatrix} (1+q\alpha)z^2 + q & 0 \\ 
0 & (1-q\delta)z^2 \end{bmatrix} + \mathcal{O}(q^2) \,.
\end{equation}
By expanding $ \det E_{-+}^q ( x ) $ as see that $ E_{-+}^q $ 
fails to be invertible when 
$ z = \pm i q^{1/2}+ {\mathcal O} ( q^{3/2} ) $. The explicit generalized 
kernel of $ F_q $ given in the beginning of \S \ref{npt} shows that the
double eigenvalue at $ 0 $ persists under perturbation.
We can now give 

\medskip

\noindent
{\em Proof of Lemma \ref{l:B}:}
We only need to check the properties of $ w_q^\pm $. The 
equation $ \sigma_3 w_{q}^{\pm } = w_{q}^{\mp} $ follows from the fact that
$ \sigma_3 F_q \sigma_3 = -F_q $. Since the eigenfuctions are simple
and $ F_q $ commutes with $ u ( x) \mapsto u ( -x ) $, and because of the 
$ \sigma_3 $ symmetry, they are either
both odd or both even. Schur's formula (see for instance \cite[\S 1]{SZ}) 
shows that
\begin{equation}
\label{eq:Bres} \Res_{z = \lambda_q^\pm } ( F_q - z ) ^{-1} = 
\Res_{z = \lambda_q^\pm } E_+^q ( z ) E_{-+}^q ( z ) ^{-1} E_{-}^q ( z ) \,, 
\end{equation}
and we note that 
\begin{equation}
\label{eq:BIm} \CC \cdot w_q^\pm =  {\rm Image} \, \Res_{z = \lambda_q^\pm } ( F_q - z ) ^{-1} \,.
\end{equation}
In addition to \eqref{eq:Eq} we also have, using \eqref{E:E+} and \eqref{E:E-}, 
\[ \begin{split} 
E_+^q ( z)\begin{bmatrix} \zeta_1 \\ \zeta_2 \end{bmatrix}  & = \begin{bmatrix} -\eta' & z(\eta+x\eta') \\ zx\eta & \eta \end{bmatrix}  \begin{bmatrix} \zeta_1 \\ \zeta_2 \end{bmatrix}
 + {\mathcal O}_H  ( q | \zeta |_{ \CC^2} ) \,, \\
E_-^q(z) \begin{bmatrix} u_1 \\ u_2 \end{bmatrix} 
& = \begin{bmatrix} -z\int u_1x\eta -\int u_2 \eta' \\ 
z\int u_2 (x\eta)' - \int u_1 \eta \end{bmatrix} + {\mathcal O}_{\CC^2}  \left( q  
\left \| u\right \|_{H} \right) 
\,. 
\end{split}\] 
This, \eqref{eq:Bres}, and 
\eqref{eq:BIm}  show that 
\[ w_q^\pm =  \begin{bmatrix}   \eta' \\ 
\pm i q^{1/2} x \eta \end{bmatrix} + {\mathcal O}_{H}  \left( q  \right)  \,. \]
Hence $ w_q^\pm $ is approximately odd, and consequently odd.
\stopthm


\section{The system of equations for $A$, $B$, $C$, $D$}
\label{AA}

Here we describe how to solve for the coefficents $A(k)$, $B(k)$, $C(k)$, and $D(k)$ in \eqref {eq:psi0}.  Define
$$\begin{bmatrix} f(x,k) \\ g(x,k) \end{bmatrix} \defeq \psi(x,k)$$
Set $\tilde A = e^{(ik-\mu)\theta}A$, $\tilde B = e^{2ik\theta}B$, $\tilde C=C$, and $\tilde D = e^{(ik-\mu)\theta}D$, 
$ \theta = \tanh^{-1} q $.
  Denote $f(0\pm) = \lim_{x \rightarrow 0\pm } f(x)$, etc.
Using that $s(0)=(1-q^2)^{1/2}$ and $t(\pm 0)=\pm q$, we obtain
\begin{align*}
e^{ik\theta} f(0-) &= (q+ik)^2 - \tilde A (1-q^2) \\
e^{ik\theta} f(0+) &= \tilde B(q-ik)^2 + \tilde C(q+ik)^2-\tilde D (1-q^2) \\
e^{ik\theta} g(0-) &=-(1-q^2)+\tilde A (q+\mu)^2 \\
e^{ik\theta} g(0+) &=-\tilde B(1-q^2)-\tilde C(1-q^2)+\tilde D (q+\mu)^2
\end{align*}

The two equations we obtain by requiring continuity at $x=0$ are
\begin{equation}
\label{E:compat1}
f(0)\defeq f(0-)=f(0+), \quad \text{and}\quad g(0)\defeq g(0-)=g(0+) \,.
\end{equation}

We further compute, from the formula for $\psi$, that
\begin{align*}
e^{ik\theta} f'(0-) & = (q+ik)^2 ik - 2(1-q^2)(q+ik) + \tilde A \Big(-(1-q^2)\mu -2(1-q^2)q\Big) \\
e^{ik\theta} f'(0+) & = \tilde B\Big( (q-ik)^2 ik + 2(1-q^2)(q-ik)\Big) \\
& \ \ \ + \, \tilde C\Big( -(q+ik)^2 ik+2(1-q^2)(q+ik) \Big) 
\\
& \ \ \ + \, \tilde D \Big( \mu(1-q^2) + 2(1-q^2)q\Big) \\
e^{ik\theta} g'(0-) & =\Big( -2(1-q^2)q-(1-q^2)ik \Big) + 
\tilde A \Big( -2(q+\mu)(1-q^2)+\mu(q+\mu)^2\Big) \\
e^{ik\theta} g'(0+) & = \tilde B \Big( 2(1-q^2)q - (1-q^2)ik \Big) + 
\tilde C\Big( 2(1-q^2)q+(1-q^2)ik\Big) \\
 & \ \ \ + \, \tilde D \Big( 2(1-q^2)(q+\mu) - \mu(q+\mu)^2\Big)
\end{align*}

The form of the derivative compatibility conditions is:
\begin{equation}
\label{E:compat2}
f'(0-)-f'(0+) = 2qf(0) \,, \ \ \
g'(0-)-g'(0+) = 2qg(0) \,. 
\end{equation}

The four equations \eqref{E:compat1},\eqref{E:compat2} give rise to the 4$\times$4 system 
\begin{equation}
 {\mathcal A} ( k, q )  
\begin{bmatrix}
\tilde A  \\ \tilde B-1 \\ \tilde C \\ \tilde D  
\end{bmatrix}
= q \begin{bmatrix}
4ik  \\  0 \\ -2(\mu-q)   \\ - 2(1-q^2)  \end{bmatrix} \; ,
\end{equation}
with the coefficient matrix $\mathcal{A}(k,q)$ 
given by
$$
\begin{bmatrix}
1-q^2 & (q-ik)^2 & (q+ik)^2 & -(1-q^2) \\
(q+\mu)^2 & 1-q^2 & 1-q^2 & -(q+\mu)^2 \\
(1-q^2) & (q-ik)(\mu-q) & (q+ik)(\mu-q) & (1-q^2) \\
-(q+\mu)(k^2+q^2) & (1-q^2)(q-ik) & (1-q^2)(q+ik) & -(q+\mu)(k^2+q^2)
\end{bmatrix} \,.
$$

\subsection{Exact solutions}
The solution is obtained from {\tt Mathematica} or, 
in principle, by Gaussian elimination.  Recalling that $\theta=\tanh^{-1}q$, we have

\begin{align*}
A&=-\frac{2 e^{(-i k+\mu ) \theta} q (i k+q) \left(-1+q^2\right)}{1+q^2 \left(-2+k^2+2 q^2\right)+2 q \left(k^2+q^2\right) \mu +\left(k^2+q^2\right) \mu ^2}\\
B&=\frac{e^{-2 i k \theta} (k-i q) (i+k (q+\mu )-i q (2 q+\mu )) (k (q+\mu )-i (1+q \mu ))}{k \left(1+q^2 \left(-2+k^2+2 q^2\right)+2 q \left(k^2+q^2\right) \mu +\left(k^2+q^2\right) \mu ^2\right)} \\
C&=-\frac{i q \left(-1+\left(2+k^2\right) q^2+2 q \left(k^2+q^2\right) \mu +\left(k^2+q^2\right) \mu ^2\right)}{k \left(1+q^2 \left(-2+k^2+2 q^2\right)+2 q \left(k^2+q^2\right) \mu +\left(k^2+q^2\right) \mu ^2\right)}\\
D&=\frac{2 e^{(-i k+\mu ) \theta} q (i k+q) \left(-1+q^2\right)}{1+q^2 \left(-2+k^2+2 q^2\right)+2 q \left(k^2+q^2\right) \mu +\left(k^2+q^2\right) \mu ^2}
\end{align*}

The numerator in the expression for $B$ is $(k-iq)v(k)w(k)$, where
$$v(k) = (i+k\mu) + q(k-i\mu) - 2iq^2, \quad 
w(k)=(-i+k\mu)+q(k-i\mu) \,.$$
We clearly see that $k=iq$ is a root of $B$, and we further find that at $k=iq$,
$$A(iq) = 0, \quad B(iq)=0, \quad C(iq)=1,\quad D(iq)=0 \,.$$
Thus,
$$\psi(x) = \begin{bmatrix}
(\tanh(|x|+\theta)-q)^2 \\
-\sech^2(|x|+\theta)
\end{bmatrix}
e^{q(|x|+\theta)}$$
solves the equation
$$H_q \psi = (1-q^2)\psi \,,$$
giving an eigenvalue when $q<0$.

We will now specify a branch of $\mu=\sqrt{2+k^2}$, and study the roots of $v(k)$ and $w(k)$ to check for consistency with Appendix \ref{S:pnz}.  Since $2+k^2$ has roots at $\pm i\sqrt 2$, we will cut along the imaginary axis, and take $\mu$ as the branch defined on the domain
$$\mathbb{C} \backslash \; (-i\infty, -i\sqrt 2] \cup [i\sqrt 2,+i\infty) \,,$$
that is real and positive for $k>0$.

We now examine $v(k)$ for $0<|q|\ll 1$.  Setting $k=-i+\kappa q^{1/2}$, we find that $\mu = 1-i\kappa q^{1/2} +\kappa^2q + \mathcal{O}(q^{3/2})$.  Substituting yields
$$v(k) = -2i(\kappa^2+1)q + \mathcal{O}(q^{3/2})\, ,$$
and thus a root occurs at $\kappa = \pm i$, i.e. when $k=-i\pm iq^{1/2}+ \mathcal{O}(q)$.  Substituting $k=-i\pm iq^{1/2}$ into the numerator of the formula for $B$, we obtain $\mathcal{O}(q^{3/2})$, while substituting into the denominator, we obtain $\mathcal{O}(q)$, and thus we have found an approximate root of $B$.  This implies that we have eigenvalues at $1+k^2= \pm 2 q^{1/2}+\mathcal{O}(q)$.  The roots of $w(k)$ occur near $k=+i$, giving nonphysical poles of the resolvent $(H_q-(k^2+1))^{-1}$.

From the above formulas, we have

\begin{align*}
\frac{A}{B} &= -\frac{2 i e^{(i k + u) \theta}k q \left(-1+q^2\right)}{(i+k (q+\mu )-i q (2 q+\mu )) 
(k (q+\mu )-i (1+q \mu ))}\\
\frac{1}{B} &= \frac{e^{2 i k \theta} k \left(1+q^2 \left(-2+k^2+2 q^2\right)+2 q 
\left(k^2+q^2\right) 
\mu +\left(k^2+q^2\right) \mu ^2\right)}{(k-i q) (i+k (q+\mu )-i q (2 q+\mu )) (k 
(q+\mu )-i (1+q \mu ))} \\
\frac{C}{B} &= -\frac{i  e^{2 i k \theta}  q \left(-1+\left(2+k^2\right) q^2+2 q 
\left(k^2+q^2\right) 
\mu +\left(k^2+q^2\right) \mu ^2\right)}{(k-i q) (i+k (q+\mu )-i q (2 q+\mu )) (k 
(q+\mu )-i (1+q \mu ))}\\
\frac{D}{B}&=\frac{2 i e^{(i k + u) \theta}   k q \left(-1+q^2\right)}{(i+k (q+\mu )-i q (2 q+\mu )) 
(k (q+\mu )-i (1+q \mu ))}
\end{align*}

\subsection{Behaviour for large $ k $.} 
\label{blk}
The 
behaviour for large values of $ k $ 
could be deduced from general principles of scattering theory. Here
we proceed directly using the matrix $ {\mathcal A} ( k , q ) $ which 
we write as 
$ {\mathcal A} ( k, q ) = {\mathcal A}_0(k) + q {\mathcal B}(k,q) $, where
\[ {\mathcal A}_0 (k)= 
\begin{bmatrix}
1 & -k^2 &  - k^2 & -1 \\
\mu^2 & 1 & 1 & - \mu^2  \\
1 &  - i k \mu & i k \mu & 1 \\
- \mu k^2 & -ik & ik &  -\mu k^2
\end{bmatrix}\,,\]
and 
\[    {\mathcal A}_0^{-1} =  \frac{1}{2 ( 1 + k^2)^2} \begin{bmatrix}
1 & k^2 & 1 & - \mu \\
-\mu^2 &  1 & ik\mu & i/k \\
- \mu^2 & 1 & -i k \mu &  -i/k \\
-1 &  -k^2 &  1 & - \mu \end{bmatrix} \,, 
\]
$ \mu = \sqrt{2+k^2} $.
For $ | k | > \epsilon > 0 $, we have 
\[  {\mathcal A}_0^{-1} 
=  {\mathcal O}_{\CC^4 \rightarrow \CC^4} ( 1/\la k\ra ^2 )\,, \  \ 
 \mathcal B = 
 {\mathcal O}_{ \CC^4 \rightarrow \CC^4 } ( \la k \ra^2 ) \,, \]
with the implicit constant in the first estimate dependent on 
$ \epsilon $.  Hence
\[  q {\mathcal A} 
_0^{-1}  {\mathcal B} = 
 {\mathcal O}_{\CC^4 \rightarrow \CC^4} ( q ) \,.  \]
For $ q $ small enough, 
depending of $\epsilon$, we can used the Neumann series inversion of 
$  I + q {\mathcal A}_0^{-1} \mathcal B $ to obtain, 
and consequently, for $ |k | > \epsilon $, 
\begin{equation}
\label{eq:1}  \begin{bmatrix} 
A \\ B-1 \\ C \\ D 
\end{bmatrix} = 
q (I+q\mathcal{A}_0^{-1}\mathcal{B})^{-1}\mathcal{A}_0^{-1}
\begin{bmatrix}
4ik  \\  0 \\ -2\mu   \\ - 6  \end{bmatrix} 
= q \mathcal{A}_0^{-1}
\begin{bmatrix}
4ik  \\  0 \\ -2\mu   \\ - 6  \end{bmatrix} 
+\mathcal{O}(q^2/\langle k \rangle ) \,. 
\end{equation}

\begin{equation}
\label{eq:largek} \begin{bmatrix} 
A\\ B-1 \\ C \\ D 
\end{bmatrix} = 
 \frac{q}{ ( 1 + k^2)^2} 
\begin{bmatrix}
2 i k +  2 \mu \\
-3 i/k - 3 i k\mu^2  \\
3i/k-ik\mu^2 \\ 
-2 i k + 2 \mu  \end{bmatrix}
+ \mathcal{O}({q^2}/{\la k \ra})
\end{equation}
This provides the estimates needed in Lemma \ref{l:fkl}.

\end{document}